\documentclass[12pt,a4paper]{article}

\usepackage{amsfonts}
\usepackage{amsmath}
\usepackage{graphicx,graphics,color,rotating,cancel}
\usepackage[normalem]{ulem}
\DeclareGraphicsExtensions{.pdf, .jpg, .tif, .png}
\usepackage{pgfplots}
\usepackage{tikz}
\usepackage{adjustbox}
\usepackage{threeparttable}
\usepackage{tabularx}
\usepackage{hyperref}
\usepackage{float}
\usepackage{longtable}
\usepackage{placeins} 
\usepackage[utf8]{inputenc}
\usepackage{amsmath}
\usepackage{amsfonts}
\usepackage{amssymb}
\usepackage{enumitem}
\usepackage{amsthm}
\usepackage{epsfig}
\usepackage{hyperref}
\usepackage{xcolor}
\usepackage{comment}
\usepackage[english]{babel}
\newtheorem{theorem}{Theorem}
\newtheorem{lem}{Lemma}

\newtheorem{proposition}{Proposition}
\newtheorem{rmk}{Remark}
\newtheorem{definition}{Definition}
\providecommand{\pr}[1]{\left(#1\right)} %(.)
\providecommand{\pp}[1]{\left[#1\right]} %[.]
\providecommand{\set}[1]{\left\lbrace#1\right\rbrace} %{.}
\providecommand{\scal}[1]{\left\langle#1\right\rangle}%<.>
\providecommand{\dual}[1]{{_{H^{-1}}\scal{#1}_{H^{1}_0}}}
\newcommand{\Db}[1]{{\color{blue}#1}}

\begin{document}

\title{The vanishing latent heat limit of a stochastic Stefan problem : An
error estimate}
\author{ Ioana Ciotir \\
Normandie University, INSA de Rouen Normandie,\\
LMI (EA 3226 - FR CNRS 3335), 76000 Rouen, France. \vspace{.5cm} \\
Perla El Kettani\\
Aix Marseille University, Toulon University,\\
Centre de Physique Th\'eorique, CNRS, Marseille, France. \vspace{.5cm} \\
Dan Goreac\\
Faculté des sciences et de génie, Université Laval \\
Québec G1V 0A6, Canada \vspace{.5cm} \\ 
Danielle Hilhorst \\
Laboratoire de Mathématiques, Universit\'e Paris-Saclay,\\
91405 Orsay Cedex, France\vspace{.5cm} \\}
\maketitle

\begin{abstract}
	\noindent The purpose of this paper is to extend an article by Hilhorst, Mimura and Sch\"atzle \cite{HMS} about the limit as the latent heat coefficient tends to zero of a two-phase Stefan problem arising in biology. We introduce a rather general additive noise white in time and colored in space, and search for the limit of the solution of the corresponding stochastic Stefan problem as the latent heat coefficient vanishes. We first prove the existence and uniqueness of the weak solution of this problem, and then study the limit of the solution as the latent heat coefficient tends to zero. Unlike in \cite{HMS}, our method of proof is based upon an error estimate between the solution of the Stefan problem with positive latent heat and that of the Stefan problem with zero latent heat, which seems to be novel even in the deterministic case when no noise is added.
\end{abstract}


\section{Introduction}

We consider the family of two-phase Stefan problems 
\begin{equation*}
(SP)\quad \quad \left\{ 
\begin{array}{ll}
db_{\sigma }(u_{\sigma } {(t)})\ni \pr{\Delta {D}(u_{\sigma }{(t)})+h(u_{\sigma } {(t)})} {\,dt}+{\sqrt{Q}}dW(t),
& {\text{ on}}~\mathcal{O}\times \left( 0,T\right) \times \Omega, \\ 
D\left( u_{\sigma }\right) =0, & {\text{ on}}~\partial \mathcal{O}\times
\left( 0,T\right) \times \Omega , \\ 
b_{\sigma }\left( u_{\sigma }\right) =b_{\sigma }\left( u_{0}\right)
=b_{\sigma }^{0}, &  {\text{ on}}~\mathcal{O}\times \left\{ 0\right\}\times \Omega ,
\end{array}%
\right. 
\end{equation*}
\noindent where $T>0$, $\mathcal{O}\subset {\mathbb{R}}^{d}$ is a bounded open set with smooth boundary, and $\pr{\Omega ,\mathcal{F},{\mathbb{P},\mathbb{F}}}$ is a {complete}
probability {basis}. {The process} $W$ is {a cylindrical} Wiener process on $L^2(\mathcal{O})$ {of the form \(\sum_{k=1}^\infty \beta_k(\cdot)e_k\), where $e_k$ is the ${L}^2(\mathcal{O})$-orthonormal basis consisting of eigenfunctions of the homogeneous Dirichlet boundary Laplace operator. The processes $\beta_k$ are independent, standard one-dimensional Brownian motions. $Q$ is a bounded, symmetric, non-negative operator on ${L}^2(\mathcal{O})$, and of trace class}. {We assume} that the latent heat coefficient $\sigma \in (0,1)$ tends to zero,
and suppose that $b_{\sigma }$ is the set-valued function 
\begin{equation*}
b_{\sigma }\left( r\right) =\left\{ 
\begin{array}{ll}
\{r\}, & ~\text{if }r<0, \\ 
\left[ 0,\sigma \right] , & ~\text{if }r=0, \\ 
\{r+\sigma \}, & ~\text{if }r>0.%
\end{array}%
\right. 
\end{equation*}%
The subdomain on which the solution $u_{\sigma }$ vanishes coincides with the interface
between two biological populations competing for their habitat, and the
drift term $h$ encodes intraspecific competition mechanisms; see \cite{HMS}.

While biological models typically involve quadratic intraspecific
competition terms so that the function $h$ is also quadratic, our analysis
requires a Lipschitz-regularized version to address technical constraints in
the stochastic setting.  The biologically meaningful
scenario involving a quadratic growth term - while presenting a richer
dynamics - introduces analytical challenges beyond our current scope and is
reserved for future research. \\
Adding the noise term $\sqrt{Q}dW(\cdot) $ to the deterministic Stefan problem 
 permits to take into consideration all sorts of small perturbations in the biological 
 context. This can be interpreted either as an absolute measurement error or, alternatively, as a mesoscopic scaling arising from a central limit theorem, complementing the macroscopic deterministic system. We stress that a linearly scaled multiplicative noise framework does not yield bounded solutions, but only ensures non-negativity. For a more biologically representative model exhibiting boundedness, the noise should be of quadratic type; however, as with the drift term, such extensions fall outside the scope of the present work.
 
The stochastic porous media equation has been intensively studied during the last ten years (see \cite{BDPR} for an extensive overview). For the case of $R^{d}$ see \cite{Rd} and \cite{reika}. For an original perspective concerning uniqueness of the solution, see also \cite{francesco} and for the probabilistic representation see \cite{francesco2}.
The special case of stochastic versions of Stefan problem was first studied in \cite{BDP}. For further extensions presenting a convection term, our readers are referred to \cite{BII} for the deterministic case, respectively to \cite{anna} for the stochastic case. For some recent developments involving a turbulence type of noise see \cite{franco1}, \cite{franco2}. 

A different way of looking at the solution is in the spirit of \cite{IDS} by using the potential or Fitzpatrick's function, see also \cite{Edi1}. For different diffusion, one can reason in the spirit of \cite{Edi2}.

In the present case we assume the initial function $u_0 $ is smooth, and such that $- 1
\leq u_0 \leq 1 - \sigma$, which implies the bounds \[- 1 \leq b_\sigma( u_0) \leq 1
.\] The function ${D}$ is defined by 
\begin{equation*}
{D}\left( r\right) =\left\{ 
\begin{array}{cc}
d_{2}r, & ~\text{if }r\leq 0, \\ 
d_{1}r, & ~\text{if }r>0.%
\end{array}
\right.
\end{equation*}

\noindent The function $h$ is Lipschitz continuous on $\mathbb{R}$, with %
$h(0) = 0 $. {We further assume} that $M$ is a positive constant such that, for all reals $u$
and $v$, 
\begin{equation}  \label{hyph1}
	\vert h(u) - h(v)\vert \leq \sqrt{M\ \min (d_1,d_2)} \vert u - v\vert.
\end{equation}
\begin{rmk}Since ${D^{'}} \geq \min(d_1,d_2)$, it follows that 
\begin{equation*}  
	\min (d_1,d_2) \vert u - v\vert	\leq \vert D(u) - D(v)\vert  ,
\end{equation*}
which combined with the hypothesis (\ref{hyph1}) implies that 
\begin{equation}  \label{hyp1}
	\vert h(u) - h(v)\vert  \leq \sqrt{\frac{M}{\min (d_1,d_2)}} \vert {D}(u) - {D}(v)\vert, 
\end{equation}
\end{rmk}

The main purpose of this paper is to study the limiting behavior of the
solution $u_\sigma$ as the latent heat coefficient vanishes, namely $\sigma \rightarrow 0$.
Before doing so, we present existence and uniqueness proofs, as well as a
priori estimates for the weak solution of Problem $(SP)$; our assumptions are similar to those of Vallet \cite{Vallet}, but we consider a different definition of the solution and the proof is based on a Galerkin method instead of a discretization in time. For more details on this method, see \cite{IDS}.

There have been several articles discussing the
singular limit of Problem $(SP)$ as $\sigma \rightarrow 0$ in different
contexts \cite{Ciotir}, \cite{eu-conv}, \cite{eu-conv2},  \cite{franco2}, \cite{HMS}. Our method of proof {exhibits} an explicit error estimate, which is also new, to the best of our knowledge, in the deterministic case.
\medskip

The organization of this paper is as follows. In Section 2, we present a slightly different formulation of the Stefan problem and define a notion of weak solution. We prove the existence and uniqueness of the solution in Section 3; the proof of the existence is based on the stochastic Galerkin method, while the proof of the uniqueness follows from computations in the $H^{-1}$ topology. We derive an error estimate between the solution to the Stefan problem with positive latent heat and that of the Stefan problem with zero latent heat in Section 4. This result seems to be novel even in the deterministic case -- when no noise is added -- as discussed in Section 5. 

\section{The stochastic Stefan problem}

It is handy to perform the change of unknown function $X_{\sigma } =
b_{\sigma }\left( u\right)$ in Problem (SP); we then obtain the initial
value problem 
\begin{equation*}
( P) \quad \quad \left\{%
\begin{array}{ll}
dX_{\sigma }{(t)}=\left( \Delta {D}\left( b_{\sigma }^{-1}\left( X_{\sigma
}{(t)}\right) \right) +h\left( b_{\sigma }^{-1}\left( X_{\sigma }{(t)}\right) \right)
\right) dt+ {\sqrt{Q}}dW{(t)}, & {\text{on }}\mathcal{O}\times \left( 0,T\right) , \\ 
D\left( b_{\sigma }^{-1}\left( X_{\sigma }\right) \right) =0, & {\text{on }}
\partial \mathcal{O}\times \left( 0,T\right) , \\ 
X_{\Db{\sigma}}\left( 0\right) =b_\sigma^0, & {\text{on }}\mathcal{O}.%
\end{array}
\right.
\end{equation*}
\noindent We {note} that the differential equation in Problem $(P)$ is a
stochastic nonlinear diffusion equation where 
\begin{equation*}
b_{\sigma }^{-1}\left( r\right) =\left\{ 
\begin{array}{ll}
r, & r<0, \\ 
0, & r\in \left[ 0,\sigma \right] , \\ 
r-\sigma , & r>\sigma.%
\end{array}
\right.
\end{equation*}%
Next we introduce some notations 
\begin{equation*}
{\cal D}_\sigma\left( r\right) := {D}\left( b_{\sigma }^{-1}\left( r\right)
\right) =\left\{ 
\begin{array}{ll}
d_{2}r, & r<0, \\ 
0, & r\in \left[0,\sigma \right] , \\ 
d_{1}\left( r-\sigma \right) , & r>\sigma ,%
\end{array}
\right.
\end{equation*}%
\begin{equation*}
h_{\sigma }\left( r\right) :=h\left( b_{\sigma }^{-1}\left( r\right) \right)
=\left\{ 
\begin{array}{ll}
h\left( r\right) , & r<0, \\ 
0, & r\in \left[ 0,\sigma \right] , \\ 
h\left( r-\sigma \right) , & r>\sigma.%
\end{array}
\right.
\end{equation*}%
\noindent We further define 

\begin{equation}  \label{defG}
G_\sigma(s) := \int_{0}^s \sqrt{\mathcal{D}^{\prime }_\sigma (\tau)} d \tau,\text{ for } s\in \mathbb{R}.
\end{equation}

The inequality \eqref{hyp1} implies that 
\begin{equation}  \label{hsigma}
\vert h_\sigma({\tilde u}) - h_\sigma({\tilde v}) \vert^2 \leq M \vert {\cal D}_\sigma({\ \tilde u}) - {\cal D}_\sigma({\ \tilde v})\vert . \vert b_{\sigma
}^{-1} ({\ \tilde u}) - {b_{\sigma }^{-1} (\tilde v) }\vert
\end{equation}
for all reals $\tilde u$ and $\tilde v$. \noindent We now rewrite Problem $%
(P)$ in the form 
\begin{equation*}
(P_\sigma) \quad \quad \left\{ 
\begin{array}{ll}
dX_{\sigma }{(t)}=\pp{\Delta {\cal D}_\sigma\left( X_{\sigma }{(t)}\right) +h_{\sigma
}\left( X_{\sigma }{(t)}\right)}dt+ {\sqrt{Q}}dW{(t)}, & {\text{ on }}\mathcal{O}\times \left(
0,T\right) , \\ 
X_{\sigma }=0, & {\text{ on }}\partial \mathcal{O}\times \left( 0,T\right) , \\ 
X_{\sigma }\left( 0\right) =b_{\sigma }^{0}, & {\text{ on }}\mathcal{O}.%
\end{array}
\right.  
\end{equation*}

We define a solution as follows. 

\begin{definition}
\label{weak def}  The function $X_\sigma$ is a solution of Problem $%
(P_\sigma)$ if 

\begin{enumerate}
\item[(i)] $X_\sigma \in L^\infty (0,T; L^2( \Omega \times \mathcal{O})) $, $
{\cal D}_\sigma\left( X_{\sigma } \right) \in L^2(\Omega \times (0,T); H^1_0(%
\mathcal{O}))$;

\item[(ii)] Almost surely and for a.e. $t \in (0,T)$  
\begin{equation}\label{def-form}
{X_{\sigma}}(t)  = b^0_{\sigma} +\Delta \int_0^t {\cal D}_\sigma\left( X_{\sigma}(s)\right) 
ds + \int_0^t  h_{\sigma }  \left( X_{\sigma}(s) \right) 
ds+ {\sqrt{Q}} W{(t)} \, \, \mathnormal{in} \  H^{-1}(\mathcal{O}).
\end{equation}
\end{enumerate}
\end{definition}

 For the sake of completeness, we recall the definition of the scalar product in ${H^{-1}} (\mathcal{O})$.
	\begin{equation*}
	<f,g>_{-1} = \int_{\mathcal O} \nabla \varphi(x) \,\,  \nabla \psi(x) \, dx, 
\end{equation*}
where 
\begin{equation*}
	-\Delta \varphi = f \mbox{    in    } \mathcal{O}, \, \varphi = 0  \mbox{    on  }  \partial   \mathcal{O}, 
\end{equation*}
and 
\begin{equation*}
	-\Delta \psi = g \mbox{    in    } \mathcal{O}, \, \psi = 0  \mbox{    on  }  \partial   \mathcal{O}.
\end{equation*}

\medskip

\noindent In turn it implies that $<- \Delta f, g >_{-1} = \int_{\mathcal{O}} f(x)g(x) dx$, for
$f\in H_0^1(\mathcal{O}),\ g\in L^2(\mathcal{O})$. \newline

Let us further mention that $H_0^1(\mathcal{O})\subset L^2(\mathcal{O})\subset H^{-1}(\mathcal{O})$, and that one frequently uses the duality product $\dual{f,g}$ between  the spaces $H^{-1} (\cal O)$ and $H_0^1(\cal O)$. In particular,  there holds  \[\dual{f,g}=\scal{f,g}_2:=\int_{\mathcal{O}}f(x)g(x)\, dx,\]for all $f\in L^2(\mathcal{O})$ and all $g\in H_0^1(\mathcal{O})$.

For convenience, the dependency on $x$, as well as other arguments, is going to be dropped in the subsequent estimates.

\section{Existence and uniqueness of the weak solution of Problem $(P_\sigma)$}
\begin{theorem}
Under the assumptions above, Problem $(P_\sigma)$  has a unique solution in the sense of the definition above.
\end{theorem}

\subsection{Proof of the existence}

In order to prove the existence of a solution of Problem $(P_\sigma)$, we
discretize the parabolic equation in Problem $(P_\sigma)$ by means of the
Galerkin method.\newline
\noindent Denote by $0< \gamma_1 < \gamma_2 \leq ... \leq \gamma_k \leq...$
the eigenvalues of the operator $-\Delta$ with homogeneous Dirichlet
boundary conditions, and by ${e}_{k}, {k} = 1,...$ the corresponding unit
eigenfunctions in $L^2( \mathcal{O})$.  Note that these are smooth functions and in particular, they belong to $H_0^1(\mathcal{O})$. Since the equation is well posed in $H^{-1}(\mathcal{O})$ we set $f_k=\gamma_k^{1/2}e_k$ so that $\{f_h\}$ is an orthonormal basis in $H^{-1}( \mathcal{O})$.

{In particular, we use the finite-dimensional projection operator 
\begin{align*}& P_m:L^2(\mathcal{O})\longrightarrow span\set{e_1,\ldots, e_m} = H_m\subset L^2(\mathcal{O}),\\
& P_m X=\sum_{k=1}^m\scal{X,e_k}_2e_k=\sum_{k=1}^m \left( \int_{\mathcal{O}}X(x)e_k(x)\, dx \right)  e_k,\qquad \mbox{ for all } X \in L^2(\mathcal{O}).
\end{align*}

Furthermore, we also project the Wiener process and define the finite-dimensional multidimensional Brownian motion\[W_m(\cdot):=P_m W(\cdot)=\sum_{k=1}^m\scal{W(\cdot),e_k}_2e_k=\sum_{k=1}^m\beta_k(\cdot)e_k.\]
To avoid tedious expressions, we will further consider the special case in which the trace operator $Q$ is linked to the Laplace one, for instance $Q=(-\Delta)^{-\theta}$, with $\theta>\frac{d+2}{2}$, which amounts to
\[\sqrt{Q}h=\sum_{k\ge 1}\sqrt{\lambda_k}\scal{h,e_k}_2 e_k,\ \lambda_k=\gamma_k^{-\theta}.\]
As a by-product, $Q$ and $\sqrt{Q}$ commute with $P_m$, and
\[P_m\pr{\sqrt{Q}W(t)}=\sqrt{Q}\pr{P_mW(t)}=\sum_{k=1}^m \sqrt{\lambda_k}\beta_k(t)e_k.\]

\noindent We look for a solution of the form 
\begin{equation*}
X_m(x, t) =\sum_{k=1}^m \left(\int_{\mathcal{O}} X_m(x,t) e_k(x) dx\right) {e}_k (x)= \sum_{k=1}^m X_{km}(t) {e}_k (x),
\end{equation*}
where $X_{km}(t)= \int_{\mathcal{O}} X_m(x,t) e_k(x) dx$. For each $m\in \mathbb{N}$ we consider the following integrated in time system of stochastic ordinary differential equations for the functions $X_{km}(t), \, k=1,2,...,m,$
\begin{equation}  \label{pbinm}
	\begin{split}
		\int_{\mathcal{O}} X_m(x,t) e_k(x) dx 
		=& \int_ \mathcal{O}%
		b^0_{\sigma m}(x) e_k(x) dx + \int_0^t 
		\dual{{{\Delta P_m {\cal D}_\sigma(X_m(s)),e_k}}} ds \\&+ 
		\int_0^t \int_{\mathcal{O}} P_m h_\sigma(X_m(x,s)) e_k(x)dxds+
		\sqrt{\lambda_k}\beta_k(t),
	\end{split}
\end{equation}
a.s. for all $t \in (0,T]$ and for all $k = 1, ..., m$,  where the initial function is defined by 
\begin{equation}  \label{pbinm0}
b^0_{\sigma m} = \sum_{k=1}^m (\int_\mathcal{O} ( b^0_{\sigma}(x) {e}_k(x)
dx) {e}_k.
\end{equation}
Note that a.s. for all  $t \in (0,T]$,  ${\cal D}_{\sigma} ( X_m(t)) \in H_0^1 (\mathcal O)$, so that   $\Delta  {\cal D}_{\sigma} ( X_m(t)) \in H^{-1}(\mathcal O)$, and that $P_m {\cal D}_\sigma(X_m(t))$  and 
$ \Delta P_m {\cal D}_\sigma(X_m(t))$ are smooth functions of the space variable. \\

Integration by parts in \eqref{pbinm} yields the equivalent integral equations
\begin{equation} \label{star} 
\begin{split}
\int_{\mathcal{O}} X_m(x,t) e_k(x) dx 
=& \int_ \mathcal{O}
b^0_{\sigma m}(x) e_k(x) dx + \int_0^t  \int_{\mathcal{O}}
		P_m {\cal D}_\sigma ( X_m (x,s)) \Delta e_k(x) dx  ds \\&+ 
\int_0^t \int_{\mathcal{O}} P_m h_\sigma(X_m(x,s)) e_k(x)dxds+
\sqrt{\lambda_k}\beta_k(t),
\end{split}
\end{equation}
a.s. for all  $t \in (0,T]$ and for all $k= 1, ..., m$.}

Problem (\ref{star}) coresponds to an initial value problem
for a system of $m$ ordinary stochastic differential equations with the
unknown functions $X_{km}(t)$, $k=1,..,m$, where the nonlinear functions 
 $ P_m  {\cal D}_\sigma$ and  $P_m h_\sigma$ are Lipschitz continuous functions 
 of $X_m$. 
It is easily seen that we are in the situation of Theorem 3.1.1 from \cite{PR} which implies that problem (\ref{star}) has a unique continuous strong solution. See also \cite{Karatzas} and \cite{DZ} for other classical results. \\

Before searching for a priori estimates, we recall the definition of the projection on the space 
$H^{-1}(\mathcal O)$ (cf. Prévôt-Röckner \cite{PR} Formula (4.2.20) p. 84), namely
\begin{equation*} 
	P_m A := A_m := \sum_{k=1}^m A_{k} {e}_k 
	=  \sum_{k=1}^m  \dual{ A, {e}_k} {e}_k, 
\end{equation*}
for all $A \in H^{-1} (\mathcal O)$, and remark that this notion of a projection operator extends the projection 
operator on $L^2(\mathcal O) $ presented above. \\

Finally we recall some useful properties of projection operators. We have that (see, e.g., \cite{EK}) 
\begin{equation*} 
	\int_{\cal O} (P_m A){(x)} X_m {(x)\, dx} = \int_{\cal O}  A X_m \mbox{ for all } A \in L^2(\cal O),
\end{equation*}
\begin{equation*} 
	\Vert P_m A \Vert_{{{L^2(\mathcal{O})}}} \leq \Vert A \Vert_{{{L^2(\mathcal{O})}}} \mbox{ for all } A \in L^2(\cal O),
\end{equation*}
and that 
\begin{equation}\label{conv}
	P_m a \to a ~~\mbox{in}~~ {L^2(\mathcal{O})} ~~\mbox{as}~~ m \to \infty.
\end{equation}

\medskip

\noindent Moreover, the following equalities hold 
\begin{lem} \label{Pmdelta}  
	We have that 
	
	\begin{enumerate}
		\item[(i)] ${\int_{\cal O} (P_m A)(x) X_m(x)\, dx = \dual{A, X_m}} \mbox{ for all } A \in H^{-1} (\cal O)$;
		
		\item[(ii)] $P_m \{\Delta A\}\ = \Delta  \{P_m A\}   \mbox{  for all  } A \in H^1(\cal O) $; 
		
		\item[(iii)] $ P_m  \{ {\Delta^{-1}} A\}  = { \Delta^{-1}}  \{ P_m A\} \mbox{  for all  } A \in L^2(\cal O)$; 
		
		\item[(iv)]  $ \int_{\cal O} \{P_m A\}(x) B(x)\, dx =  	\int_{\cal O} A(x) \{P_m B\}(x)\, dx  \mbox{  for all  } A, B \in L^2(\cal O)$. 
	\end{enumerate}
	
\end{lem}
\noindent We refer to the Appendix for the proof of these results.\\

Next, we recall an Itô formula based on Da Prato and Zabczyk \cite{DZ}, p. 106. While it is meant for SPDEs, we apply it to the approximate solution, which we handle as if it was the solution of a SPDE.

\begin{lem} [Itô formula] \label{Itof}  
		Let $X$ be an $E$-valued function such that \[ X(t) = X(0) + \int_0^t h(s)\,ds + \int_0^t G(s)\,dW(s), \qquad 0\leq s\leq t, \] and suppose that $h$ is an $E$-valued predictable process Bochner integrable on $[0,T]$, a.s., $G$ is an $E$-valued process stochastically integrable. Suppose that the function \[ F:[0,T]\times E\to\mathbb{R} \] and its partial derivatives \[ \frac{\partial F}{\partial t}, \qquad \frac{\partial F}{\partial X}, \qquad \frac{\partial^2 F}{\partial X^2} \] are continuous on $[0,T]\times E$.
		Then for all $t\in[0,T]$, a.s., \[ \begin{aligned} F(t,X(t)) ={}&F(0,X(0)) +\int_0^t \frac{\partial F}{\partial t}(s,X(s))\,ds\\ &+\int_0^t \left\langle \frac{\partial F}{\partial X}(s,X(s)), h(s) \right\rangle_E\,ds\\ &+\int_0^t \left\langle \frac{\partial F}{\partial X}(s,X(s)), G(s)\,dW(s) \right\rangle_E\\ &+\frac12\int_0^t \operatorname{Tr} \left[ \frac{\partial^2F}{\partial X^2}(s,X(s)) \left(G(s)Q^{\frac12}\right) \left(G(s)Q^{\frac12}\right)^* \right]\,ds, \end{aligned} \tag{30} \]
where
	\[ \begin{aligned} \operatorname{Tr}\left[ \frac{\partial^2F}{\partial X^2}(s,X(s)) \left(G(s)Q^{\frac12}\right) \left(G(s)Q^{\frac12}\right)^* \right] &= \sum_{k=1}^{\infty} \left< \frac{\partial^2F}{\partial X^2} \left( s,X(s) \right) \left( G(s)Q^{\frac12}e_k \right) , \left( G(s)Q^{\frac12}e_k \right) \right>_{E}, 
	\end{aligned} \]
    and $E$ is a hilbert space. 
\end{lem}

In  what follows we prove the following a priori estimates.

\begin{proposition}
\label{estXm} Let $T > 0$ be arbitrary. There exist positive constants $C$
and $\tilde{C}$ {that may depend on the time horizon $T$, but can be chosen} independent of $m$ and $\sigma$ such that 
	\begin{align*}
	&\Vert {X}_m \Vert^2_{L^\infty(0,T; L^{2}(\Omega \times \mathcal{O}))} \leq C,\text{ and}
	\\[10pt]
	&\Vert 	{\cal D}_\sigma({X}_m) \Vert^2_{L^2 (\Omega \times (0,T); H^1_0(\mathcal{O}))}
	\leq \tilde{C}. & 
\end{align*}
\end{proposition}

\noindent 
\begin{proof}
Applying Lemma \ref{Pmdelta} (ii) to the equations in the system \eqref{pbinm} yields 
\begin{equation} \label{pbinp}
	\begin{split}
		\int_{\mathcal{O}} X_m(x,t) e_k(x) dx 
		=& \int_ \mathcal{O}%
		b^0_{\sigma m}(x) e_k(x) dx + \int_0^t 
		\dual{{{ P_m \Delta{\cal D}_\sigma(X_m(s)),e_k}}} ds \\&+ 
		\int_0^t \int_{\mathcal{O}} P_m h_\sigma(X_m(x,s)) e_k(x)dxds+
		\sqrt{\lambda_k}\beta_k(t),
	\end{split}
\end{equation}
a.s. for all $t \in (0,T]$ and for all $k = 1, ..., m$.
{Multiplying each equation of \eqref{pbinp} by $e_k$ and summing for $k = 1, ...,m$, we obtain}
\begin{equation} \label{compact}
	\begin{split}
X_{m}(t) = &~ b^0_{\sigma m} { + \int_0^t P_m \Delta \mathcal{D}_\sigma (X_m(s)) ds\, }\\&+\int_0^t  P_m h_\sigma(X_m(s))ds
		+ {\sum_{k=1}^m\sqrt{\lambda_k}\beta_k(t)e_k}.
	\end{split} 
\end{equation}
a.s. for all $t \in (0,T]$.

Let us apply Ito's formula to  \eqref{compact} with { $F(X) = \Vert X \Vert _{L^2(\mathcal{O})}^2$, and then  take the expectation} to deduce that
\begin{equation*}
\begin{split}
	\mathbb{E}\int_ \mathcal{O}\displaystyle{ X^2_m(x,t)} dx \le &	\ \mathbb{E}\int_ \mathcal{O}(\displaystyle{ b^0_{\sigma m}(x)})^2 dx +2 \mathbb{E} \int_0^t \int_\mathcal{O}   P_m \Delta {\cal D}_{\sigma} ( X_m )   X_mdxds\\
    &+ 2\mathbb{E} \int_0^t \int_\mathcal{O} P_m h_\sigma(X_m) X_mdxds  + t\, \text{Tr}_{L^2}( Q),
\end{split}
\end{equation*}
where $\text{Tr}_{L^2}(Q)=\sum_{k=1}^\infty \langle Qe_k,e_k\rangle_{L^2(\mathcal{O})}=\sum_{k=1}^\infty \lambda_k=\sum_{k=1}^\infty \gamma_k^{-\theta}<\infty$. Thus
\begin{align*}
\mathbb{E}  \Vert X_m(t) \Vert^2_2 \leq  &\Vert b^0_{\sigma m} \Vert^2_{L^2} - 2\mathbb{E}  \int_0^t  \int_\mathcal{O} \mathcal \nabla {D}_{\sigma} ( X_m)  \nabla X_mdxds  + 2\mathbb{E}  \int_0^t \int_\mathcal{O} h_\sigma (X_m) X_mdxds + t \text{Tr}_{L^2}( Q)  \\ 
 =  &\Vert b^0_{\sigma m} \Vert^2_{L^2} - 2 \mathbb{E}\int_0^t  \int_\mathcal{O} \mathcal ( \sqrt {{D}'_{\sigma} (X_m) } \nabla X_m)^2dxds  + 2 \mathbb{E}\int_0^t \int_\mathcal{O} h_\sigma (X_m) X_mdxds + t \text{Tr}_{L^2}( Q)  \nonumber   \\
=& \Vert b^0_{\sigma m}\Vert^2_{L^2} - 2 \mathbb{E} \int_0^t  \int_\mathcal{O} \mathcal ( \nabla G_\sigma(X_m))^2dxds   + 2 \mathbb{E}\int_0^t \int_\mathcal{O} h_\sigma (X_m) X_mdxds + t \text{Tr}_{L^2}( Q) ,
\end{align*}
which in turn implies that
\begin{align}
&\mathbb{E} \Vert X_m(t) \Vert^2_2 + 2 \mathbb{E} \int_0^t  \int_\mathcal{O} \mathcal ( \nabla G_\sigma(X_m))^2dxds \nonumber  \\
&\leq  \Vert b^0_{\sigma m} \Vert^2_{L^2}    + 2 \mathbb{E}\int_0^t \int_\mathcal{O} h_\sigma (X_m) X_m dxds + t \text{Tr}_{L^2}( Q)  \nonumber \\
&= \Vert b^0_{\sigma m}   \Vert^2_{L^2}    + 2 \mathbb{E}\int_0^t \int_\mathcal{O} (h_\sigma(X_m) - h_\sigma(0)) X_mdxds + t \text{Tr}_{L^2}( Q)  \nonumber  \\
& \leq  \Vert b^0_{\sigma m} \Vert^2_{L^2}    + 2 M  \mathbb{E}\int_0^t \int_\mathcal{O} 
\vert {\cal D}_\sigma(X_m) - {\cal D}_\sigma(0) \vert   \vert b_\sigma^{-1}(X_m) - b_\sigma^{-1}(0) \vert dxds   + t \text{Tr}_{L^2}( Q)   \nonumber \\ 
& \leq  \Vert b^0_{\sigma m}\Vert^2_{L^2}    + 2 M \max(d_1,d_2)  \mathbb{E}\int_0^t \int_\mathcal{O} 
\vert  X_m\vert^2dxds + t \text{Tr}_{L^2}( Q)  \nonumber \\
& \leq C_1 + C_2 \mathbb{E}\int_0^t \int_\mathcal{O} 
\vert  X_m\vert^2{\,dxds}. \nonumber
\end{align} 
Applying Gronwall Lemma we obtain 
\begin{align}
&{ \Vert X_m \Vert^2_{L^{\infty}( 0,T; L^2(\Omega \times\mathcal{O}))} \leq  C_1 e^{C_2 T}, } \label{estonxm} \\[10pt]
&\mathbb{E} \Vert \nabla G_\sigma (X_m) \Vert^2_{L^2(0,T; L^2(\mathcal{O}))} \leq  C \label{estong}. 
\end{align}

From the last inequality \eqref{estong}, we deduce that 
\begin{align*}
&\mathbb{E} \int_0^t \int_{\mathcal{O}} (\nabla {\cal D}_\sigma (X_m) )^2 dx ds = \mathbb{E} \int_0^t \int_{\mathcal{O}} \vert  {\cal D}'_\sigma (X_m) )\vert  \vert \sqrt{{\cal D}'_\sigma (X_m) )} \nabla X_m \vert^2 dx ds \\
&= \mathbb{E} \int_0^t \int_{\mathcal{O}} \vert  {\cal D}'_\sigma (X_m) )\vert   \left(  \nabla(\int_0^{X_m} \sqrt{{\cal D}'_\sigma (r) } dr)  \right)^2 dx ds \\
&\leq \max(d_1,d_2) \mathbb{E} \int_0^t \int_{\mathcal{O}} \left( \nabla G_\sigma (X_m) \right)^2 dx dt \leq  \tilde{C},
\end{align*}
where $\tilde{C}=\tilde{C}(T)$. 
From \eqref{estonxm}, we can deduce that there exists a subsequence of $\{X_m\}$ which we denote again by $\{X_m\}$ and a function $X_\sigma$ such that 
\begin{align} \label{convxm}
    X_m\rightharpoonup X_\sigma ~~\mbox{weakly in}~~ L^2(\Omega \times \mathcal{O}  \times (0,T)  ) \mbox{ as}~ m \to \infty.
\end{align}
\end{proof}   
We now have to determine the limits of the nonlinear terms ${{\cal D}_\sigma}(X_m)$
and ${h_\sigma}(X_m)$ as $m \rightarrow \infty$. To that purpose, we will
prove that they strongly converge to their limits in $L^2(\Omega \times \mathcal{O} \times (0,T))$ as $m \rightarrow \infty$; in turn this will permit to identify their limits.\\
To begin with, let us recall an Itô formula in the space ${H^{-1}} (\mathcal{O})$.
\begin{lem}[It{\^{o}} formula in ${H^{-1}} (\mathcal{O})$]
\label{lem:Ito_formula}  Let $U:[0,T]\times \Omega \rightarrow {H^{-1}} (\mathcal{O})$ be a stochastic process such that  
\begin{equation*}
U(t) = U(0) + \int_0^t \Phi(s)\ ds + \int_0^t \Psi(s) dW(s) \quad\text{a.s. in } {H^{-1}} (\mathcal{O}), 
\end{equation*}
where $\Phi \in L^{1}([0,T]; L^{1}(\Omega ; H^{-1} (\mathcal{O}))$  and $\Psi \in L^{1}([0,T]; L^{1}(\Omega ; L_{2}(H^{-1} (\mathcal{O})))$ where $L_{2}(H^{-1} (\mathcal{O})))$ is the space of Hilbert-Schmidt operators from $H^{-1} (\mathcal{O}))$. Let $F$
be a twice Fr\'echet differentiable real-valued functional on ${H^{-1}} (\mathcal{O})$.
Then, the following holds :  
\begin{equation*}
F(U(t)) = F(U(0)) + \int_0^t <{DF ((U(s)),\Phi(s)}>_{-1} ds + \int_0^t <{DF  (U(s)), \Psi(s)}dW(s) >_{-1} 
\end{equation*}
\begin{equation*}
+\frac{1}{2} \int_0^t \text{ Tr }[\Psi^{*}(s)D^2 F(U(s)) \Psi (s) ]ds
\end{equation*}
where $<., .>_{-1}$ is the scalar product in ${H^{-1}} (\mathcal{O})$.
\end{lem}
\noindent(See e.g. \cite{BDPR} page 10.)

{For every $m>l$ we define}
\begin{equation}  \label{Vml}
V_{ml}(x,t) = (- \Delta)^{-1}(X_m(x,t) - X_l (x,t)),
\end{equation}
for all $(x,t) \in \mathcal{O} \times (0,T)$.

\begin{proposition}\label{P2}
{There exists a sequence $(\eta_{ml})_{m\ge l}\subset \mathbb{R}_+$, such that \[\lim_{l\rightarrow\infty}\lim_{m\rightarrow\infty}\eta_{ml}=0,\]and\[\mathbb{E} \int_\mathcal{O} \vert \nabla V_{ml}(t) \vert^2dx +
\mathbb{E} \int_0^t \int_\mathcal{O} ( {\cal D}_{\sigma} ( X_m) - {\cal D}_{\sigma} (
X_l )) ( X_m - X_l)dxds\leq\eta_{ml},\ \forall 0\le l\le m.\]}
\end{proposition}
\begin{proof}
We apply Lemma \ref{lem:Ito_formula} to the difference of the equations for $X_m$ and $X_l$, namely
\begin{equation*} \label{pbinm2}
	\begin{split}
 \displaystyle{( X_m - X_l)(x,t)}  = \displaystyle{b^0_{\sigma m} -   b^0_{\sigma l}} +  \int_0^t  ( P_m \Delta {\cal D}_{\sigma} ( X_m) - P_l \Delta {\cal D}_{\sigma} ( X_l) ) \\
 + (P_m h_\sigma(X_m) - P_l h_\sigma(X_l)))ds  + { (P_m (\sqrt{Q}W(t)) - P_l (\sqrt{Q}W(t)))},
\end{split}
\end{equation*}
which yields, by setting $F(X) = 	\Vert X	\Vert ^2_{{H^{-1}}(\cal O)}$, 
\begin{equation*}
\begin{split}
	&\Vert (X_m - X_l)(t)\Vert ^2_{{H^{-1}}(\cal O)}\le  \Vert  \displaystyle{b^0_{\sigma m} -   b^0_{\sigma l}}\Vert ^2_{{H^{-1}}(\cal O)}\\ &\qquad +2 \int_0^t <P_m \Delta {\cal D}_{\sigma} ( X_m) -P_l  \Delta {\cal D}_{\sigma} ( X_l)  + P_m h_\sigma(X_m) - P_l h_\sigma(X_l), X_m - X_l>_{-1}ds \\
    &\qquad +2 \int_0^t <X_m - X_l,P_m(\sqrt{Q}dW(s))-P_l(\sqrt{Q}dW(s))>_{-1}\,ds  + t \, { \text{Tr}_{H^{-1}}( (P_m -P_l) (Q))}.
\end{split}
\end{equation*}
Here $$\text{Tr}_{H^{-1}}( Q)=\sum_{k=1}^\infty <Qf_k,f_k>_{H^{-1}}=\sum_{k=1}^\infty <(-\Delta)^{-\theta} \gamma_k^{1/2}e_k,\gamma_k^{1/2}e_k>_{H^{-1}}=\sum_{k=1}^\infty \gamma_k^{-\theta}< \infty$$
and $$\text{Tr}_{H^{-1}}( (P_m -P_l) (Q))=\sum_{l< k\le m} \gamma_k^{-\theta}.$$

In view of \eqref{Vml}, we deduce the following inequality
\begin{eqnarray}
&&\mathbb{E}\int_{\mathcal{O}}\left\vert \nabla V_{ml}\left( t\right)
\right\vert ^{2}dx+\underset{=T_{1}}{\underbrace{2\mathbb{E}%
\int_{0}^{t}\int_{\mathcal{O}}\left( P_{m}D_{\sigma }\left( X_{m}\right)
-P_{l}D_{\sigma }\left( X_{l}\right) \right) \left( X_{m}-X_{l}\right) dxds}}
\label{v_ml1} \\
&\leq &\int_{\mathcal{O}}\left\vert \nabla V_{ml}\left( 0\right) \right\vert
^{2}dx+t \, {\sum_{l< k\le m} \gamma_k^{-\theta}} +\underset{=T_{2}}{\underbrace{2\mathbb{E}\int_{0}^{t}\int_{\mathcal{O}%
}\left( P_{m}h_{\sigma }\left( X_{m}\right) -P_{l}h_{\sigma }\left(
X_{l}\right) \right) V_{ml}dxds}}  \notag
\end{eqnarray}

Next we transform the terms $T_1$ and $T_2$ as follows.
\begin{eqnarray*}
T_{1} &=&2\mathbb{E}\int_{0}^{t}\int_{\mathcal{O}}\left( P_{m}D_{\sigma
}\left( X_{m}\right) -P_{l}D_{\sigma }\left( X_{l}\right) \right) \left(
X_{m}-X_{l}\right) dxds \\
&=&2\mathbb{E}\int_{0}^{t}\int_{\mathcal{O}}\left( P_{m}D_{\sigma }\left(
X_{m}\right) -P_{m}D_{\sigma }\left( X_{l}\right) \right) \left(
X_{m}-X_{l}\right) dxds \\
&&+2\mathbb{E}\int_{0}^{t}\int_{\mathcal{O}}\left( P_{m}D_{\sigma }\left(
X_{l}\right) -P_{l}D_{\sigma }\left( X_{l}\right) \right) \left(
X_{m}-X_{l}\right) dxds \\
&=&2\mathbb{E}\int_{0}^{t}\int_{\mathcal{O}}\left( D_{\sigma }\left(
X_{m}\right) -D_{\sigma }\left( X_{l}\right) \right) P_{m}\left(
X_{m}-X_{l}\right) dxds \\
&&+2\mathbb{E}\int_{0}^{t}\int_{\mathcal{O}}\left( P_{m}-P_{l}\right) \left(
D_{\sigma }\left( X_{l}\right) \right) \left( X_{m}-X_{l}\right) dxds.
\end{eqnarray*}

Note that we have the following properties%
\begin{equation*}
\left( P_{m}-P_{l}\right) \left( X_{l}\right) =0\text{ and }\left(
P_{m}-P_{l}\right) \left( X_{m}\right) =\left( I-P_{l}\right) \left(
X_{m}\right) .
\end{equation*}

Going back to the previous relation, we have that%
\begin{eqnarray*}
T_{1} &=&2\mathbb{E}\int_{0}^{t}\int_{\mathcal{O}}\left( D_{\sigma }\left(
X_{m}\right) -D_{\sigma }\left( X_{l}\right) \right) \left(
X_{m}-X_{l}\right) dxds \\
&&+2\mathbb{E}\int_{0}^{t}\int_{\mathcal{O}}D_{\sigma }\left( X_{l}\right)
\left( \left( P_{m}-P_{l}\right) X_{m}-\left( P_{m}-P_{l}\right)
X_{l}\right) dxds \\
&=&2\mathbb{E}\int_{0}^{t}\int_{\mathcal{O}}\left( D_{\sigma }\left(
X_{m}\right) -D_{\sigma }\left( X_{l}\right) \right) \left(
X_{m}-X_{l}\right) dxds \\
&&+2\mathbb{E}\int_{0}^{t}\int_{\mathcal{O}}D_{\sigma }\left( X_{l}\right)
\left( I-P_{l}\right) \left( X_{m}\right) dxds \\
&=&2\mathbb{E}\int_{0}^{t}\int_{\mathcal{O}}\left( D_{\sigma }\left(
X_{m}\right) -D_{\sigma }\left( X_{l}\right) \right) \left(
X_{m}-X_{l}\right) dxds \\
&&+2\mathbb{E}\int_{0}^{t}\int_{\mathcal{O}}\left( I-P_{l}\right) \left(
D_{\sigma }\left( X_{l}\right) \right) \left( X_{m}\right) dxds.
\end{eqnarray*}

Following the same idea, we have
\begin{eqnarray*}
T_{2} &=&2\mathbb{E}\int_{0}^{t}\int_{\mathcal{O}}\left( P_{m}h_{\sigma
}\left( X_{m}\right) -P_{l}h_{\sigma }\left( X_{l}\right) \right) \left(
-\Delta \right) ^{-1}\left( X_{m}-X_{l}\right) dxds \\
&=&2\mathbb{E}\int_{0}^{t}\int_{\mathcal{O}}\left( P_{m}h_{\sigma }\left(
X_{m}\right) -P_{m}h_{\sigma }\left( X_{l}\right) \right) \left( -\Delta
\right) ^{-1}\left( X_{m}-X_{l}\right) dxds \\
&&+2\mathbb{E}\int_{0}^{t}\int_{\mathcal{O}}\left( P_{m}h_{\sigma }\left(
X_{l}\right) -P_{l}h_{\sigma }\left( X_{l}\right) \right) \left( -\Delta
\right) ^{-1}\left( X_{m}-X_{l}\right) dxds \\
&=&2\mathbb{E}\int_{0}^{t}\int_{\mathcal{O}}\left( h_{\sigma }\left(
X_{m}\right) -h_{\sigma }\left( X_{l}\right) \right) \left( -\Delta \right)
^{-1}P_{m}\left( X_{m}-X_{l}\right) dxds \\
&&+2\mathbb{E}\int_{0}^{t}\int_{\mathcal{O}}h_{\sigma }\left( X_{l}\right)
\left( -\Delta \right) ^{-1}\left( P_{m}-P_{l}\right) \left(
X_{m}-X_{l}\right) dxds \\
&=&2\mathbb{E}\int_{0}^{t}\int_{\mathcal{O}}\left( h_{\sigma }\left(
X_{m}\right) -h_{\sigma }\left( X_{l}\right) \right) \left( -\Delta \right)
^{-1}\left( X_{m}-X_{l}\right) dxds \\
&&+2\mathbb{E}\int_{0}^{t}\int_{\mathcal{O}}\left( I-P_{l}\right) \left(
h_{\sigma }\left( X_{l}\right) \right) \left( -\Delta \right) ^{-1}\left(
X_{m}\right) dxds \\
&\leq &\frac{1}{\alpha }\mathbb{E}\int_{0}^{t}\int_{\mathcal{O}}\left\vert
h_{\sigma }\left( X_{m}\right) -h_{\sigma }\left( X_{l}\right) \right\vert
^{2}dxds+\alpha \mathbb{E}\int_{0}^{t}\int_{\mathcal{O}}\left\vert
V_{ml}\right\vert ^{2}dxds \\
&&+2\mathbb{E}\int_{0}^{t}\int_{\mathcal{O}}\left( I-P_{l}\right) \left(
h_{\sigma }\left( X_{l}\right) \right) \left( -\Delta \right) ^{-1}\left(
X_{m}\right) dxds.
\end{eqnarray*}

Going back to (\ref{v_ml1}), we obtain 
\begin{eqnarray*}
&&\mathbb{E}\int_{\mathcal{O}}\left\vert \nabla V_{ml}\left( t\right)
\right\vert ^{2}dx+2\mathbb{E}\int_{0}^{t}\int_{\mathcal{O}}\left( D_{\sigma
}\left( X_{m}\right) -D_{\sigma }\left( X_{l}\right) \right) \left(
X_{m}-X_{l}\right) dxds \\
&&+2\mathbb{E}\int_{0}^{t}\int_{\mathcal{O}}\left( I-P_{l}\right) \left(
D_{\sigma }\left( X_{l}\right) \right) \left( X_{m}\right) dxds \\
&\leq &\int_{\mathcal{O}}\left\vert \nabla V_{ml}\left( 0\right) \right\vert
^{2}dx+t\, {\sum_{l< k\le m} \gamma_k^{-\theta}} \\
&&+\frac{1}{\alpha }\mathbb{E}\int_{0}^{t}\int_{\mathcal{O}}\left\vert
h_{\sigma }\left( X_{m}\right) -h_{\sigma }\left( X_{l}\right) \right\vert
^{2}dxds+\alpha \mathbb{E}\int_{0}^{t}\int_{\mathcal{O}}\left\vert
V_{ml}\right\vert ^{2}dxds \\
&&+2\mathbb{E}\int_{0}^{t}\int_{\mathcal{O}}\left( I-P_{l}\right) \left(
h_{\sigma }\left( X_{l}\right) \right) \left( -\Delta \right) ^{-1}\left(
X_{m}\right) dxds.
\end{eqnarray*}

Keeping in mind the assumptions on $h$ we have that 
\begin{eqnarray*}
\frac{1}{\alpha }\mathbb{E}\int_{0}^{t}\int_{\mathcal{O}}\left\vert
h_{\sigma }\left( X_{m}\right) -h_{\sigma }\left( X_{l}\right) \right\vert
^{2}dxds \leq \frac{MC}{\alpha }\mathbb{E}\int_{0}^{t}\int_{\mathcal{O}}\left\vert
D_{\sigma }\left( X_{m}\right) -D_{\sigma }\left( X_{l}\right) \right\vert
\left\vert X_{m}-X_{l}\right\vert dxds.
\end{eqnarray*}

Since $D_{\sigma }$ is monotonically increasing and choosing $\alpha$ such that $\alpha \geq MC$ obtain
\begin{eqnarray*}
&&\mathbb{E}\int_{\mathcal{O}}\left\vert \nabla V_{ml}\left( t\right)
\right\vert ^{2}dx+\mathbb{E}\int_{0}^{t}\int_{\mathcal{O}}\left( D_{\sigma
}\left( X_{m}\right) -D_{\sigma }\left( X_{l}\right) \right) \left(
X_{m}-X_{l}\right) dxds \\
&\leq &\int_{\mathcal{O}}\left\vert \nabla V_{ml}\left( 0\right) \right\vert
^{2}dx+t \, {\sum_{l< k\le m} \gamma_k^{-\theta}} \\
&&+\alpha \mathbb{E}\int_{0}^{t}\int_{\mathcal{O}}\left\vert
V_{ml}\right\vert ^{2}dxds+2\mathbb{E}\int_{0}^{t}\int_{\mathcal{O}}\left(
I-P_{l}\right) \left( h_{\sigma }\left( X_{l}\right) \right) \left( -\Delta
\right) ^{-1}\left( X_{m}\right) dxds \\
&&+2\mathbb{E}\int_{0}^{t}\int_{\mathcal{O}}\left( P_{l}-I\right) \left(
D_{\sigma }\left( X_{l}\right) \right) \left( X_{m}\right) dxds \\
&\leq &\int_{\mathcal{O}}\left\vert \nabla V_{ml}\left( 0\right) \right\vert
^{2}dx+t\, {\sum_{l< k\le m} \gamma_k^{-\theta}} \\
&&+\alpha C\mathbb{E}\int_{0}^{t}\int_{\mathcal{O}}\left\vert \nabla
V_{ml}\right\vert ^{2}dxds+2\mathbb{E}\int_{0}^{t}\int_{\mathcal{O}}\left(
I-P_{l}\right) \left( h_{\sigma }\left( X_{l}\right) \right) \left( -\Delta
\right) ^{-1}\left( X_{m}\right) dxds \\
&&+2\mathbb{E}\int_{0}^{t}\int_{\mathcal{O}}\left( P_{l}-I\right) \left(
D_{\sigma }\left( X_{l}\right) \right) \left( X_{m}\right) dxds.
\end{eqnarray*}

We set
\begin{eqnarray*}
\theta _{l,m}\left( s\right) &=&\int_{\mathcal{O}}\left( P_{l}-I\right)
\left( D_{\sigma }\left( X_{l}\left( s\right) \right) \right) X_{m}\left(
s\right) dx \\
&&+\int_{\mathcal{O}}\left( I-P_{l}\right) \left( h_{\sigma }\left(
X_{l}\left( s\right) \right) \right) \left( -\Delta \right) ^{-1}\left(
X_{m}\left( s\right) \right) dx
\end{eqnarray*}

and we apply the Gronwall inegality in 
\begin{eqnarray*}
&&\mathbb{E}\int_{\mathcal{O}}\left\vert \nabla V_{ml}\left( t\right)
\right\vert ^{2}dx+\mathbb{E}\int_{0}^{t}\int_{\mathcal{O}}\left( D_{\sigma
}\left( X_{m}\right) -D_{\sigma }\left( X_{l}\right) \right) \left(
X_{m}-X_{l}\right) dxds \\
&\leq &\int_{\mathcal{O}}\left\vert \nabla V_{ml}\left( 0\right) \right\vert
^{2}dx+t\, {\sum_{l< k\le m} \gamma_k^{-\theta}} +\int_{0}^{t}\alpha C \mathbb{E}\int_{\mathcal{O}}\left\vert \nabla V_{ml}\right\vert ^{2}dxds+\int_{0}^{t}2\mathbb{E}\theta _{l,m}\left(
s\right) ds
\end{eqnarray*}
to obtain 
\begin{eqnarray*}
&&\mathbb{E}\int_{\mathcal{O}}\left\vert \nabla V_{ml}\left( t\right)
\right\vert ^{2}dx+\mathbb{E}\int_{0}^{t}\int_{\mathcal{O}}\left( D_{\sigma
}\left( X_{m}\right) -D_{\sigma }\left( X_{l}\right) \right) \left(
X_{m}-X_{l}\right) dxds   \\
&\leq &{e^{\alpha C
T}\pr{2\mathbb{E}\int_{0}^{t}\theta _{l,m}\left( s\right) ds+\left\{ \int_{\mathcal{O}}\left\vert \nabla V_{ml}\left( 0\right)
\right\vert ^{2}dx+T\, {\sum_{l< k\le m} \gamma_k^{-\theta}}\right\}}} .  
\end{eqnarray*}

Finally if we define
\begin{eqnarray*}
\theta _{l} &=&\int_{\mathcal{O}}\left( P_{l}-I\right) \left( D_{\sigma
}\left( X_{l}\right) \right) X_\sigma dx \\
&&+\int_{\mathcal{O}}\left( I-P_{l}\right) \left( h_{\sigma }\left(
X_{l}\right) \right) \left( -\Delta \right) ^{-1}\left( X_\sigma \right) dx.
\end{eqnarray*}

For each fixed $l$, by Proposition 1 and the weak convergence of $\left\{ X_{m}\right\} $ in $%
L^{2}\left( \left( 0,T\right) \times \Omega \times \mathcal{O}\right) $ we have that%
\begin{equation*}
\underset{m\rightarrow \infty }{\lim }\mathbb{E}\int_{0}^{t}\theta
_{l,m}\left( s\right) ds=\mathbb{E}\int_{0}^{t}\theta _{l}\left( s\right) ds.
\end{equation*}
We also have that
\begin{eqnarray*}
\underset{m\rightarrow \infty }{\lim }e^{\alpha CT}\left\{ \int_{\mathcal{O}%
}\left\vert \nabla V_{ml}\left( 0\right) \right\vert ^{2}dx+T\, {\sum_{l< k\le m} \gamma_k^{-\theta}}\right\} \leq & C(T) \left\{\int_{\mathcal{O}}\left\vert b_{\sigma }^{0}-b_{\sigma
l}^{0}\right\vert ^{2}dx+{\sum_{l< k} \gamma_k^{-\theta}}\right\}.
\end{eqnarray*}
where $C(T)$ is some positive constant which only depends on $T$. We note that
\begin{eqnarray*}
\mathbb{E}\int_{0}^{t}\theta _{l}\left( s\right) ds &=&\mathbb{E}
\int_{0}^{t}\int_{\mathcal{O}}D_{\sigma }\left( X_{l}\right) \left(
P_{l}-I\right) X_\sigma dxds \\
&&+\mathbb{E}\int_{0}^{t}\int_{\mathcal{O}}h_{\sigma }\left( X_{l}\right)
\left( -\Delta \right) ^{-1}\left( I-P_{l}\right) \left( X_\sigma \right) dxds,
\end{eqnarray*}
and 
\begin{eqnarray*}
\left\vert \mathbb{E}\int_{0}^{t}\theta _{l}\left( s\right) ds\right\vert
&\leq &\left\vert \mathbb{E}\int_{0}^{t}\int_{\mathcal{O}}D_{\sigma }\left(
X_{l}\right) \left( P_{l}-I\right) X_\sigma dxds\right\vert \\
&&+\left\vert \mathbb{E}\int_{0}^{t}\int_{\mathcal{O}}h_{\sigma }\left(
X_{l}\right) \left( -\Delta \right) ^{-1}\left( I-P_{l}\right) \left(
X_\sigma \right) dxds\right\vert \\
&\leq &\mathbb{E}\int_{0}^{t}<D_{\sigma }\left(
X_{l}\right), \left( P_{l}-I\right) X_\sigma >_{L^{2}\left( \mathcal{O}\right) } ds \\
&&+\mathbb{E}\int_{0}^{t}<h_{\sigma }\left(
X_{l}\right), \left( I-P_{l}\right) \left(
X_\sigma \right) > _{H^{-1}\left( \mathcal{O}\right) }ds\\
&\leq &\mathbb{E}\int_{0}^{t}\left\Vert D_{\sigma }\left( X_{l}\right)
\right\Vert _{L^{2}\left( \mathcal{O}\right) }\left\Vert \left(
P_{l}-I\right) X_\sigma \right\Vert _{L^{2}\left( \mathcal{O}\right) }ds \\
&&+c\, \mathbb{E}\int_{0}^{t}\left\Vert h_{\sigma }\left( X_{l}\right)
\right\Vert _{H^{-1}\left( \mathcal{O}\right) }\left\Vert \left(
I-P_{l}\right) \left( X_\sigma \right) \right\Vert _{L^{2}\left( \mathcal{O}\right)
}ds
\end{eqnarray*}
where $c$ is a positive constant. Since 
\begin{equation*}
\underset{l\rightarrow \infty }{\lim }\left( I-P_{l}\right) \left( X_\sigma \right)
=0\text{ in }L^{2}\left( \left( 0,T\right) ;L^{2}\left( \Omega ;L^{2}\left( 
\mathcal{O}\right) \right) \right)
\end{equation*}%
we have that $\underset{l\rightarrow \infty } {\lim } \left\vert \mathbb{E} \displaystyle{\int_{0}^{t}} \theta
_{l}(s) ds\right\vert =0$.

Next we set $$\eta_{ml}= {e^{\alpha C
T}\pr{2\mathbb{E}\int_{0}^{t}\theta _{l,m}\left( s\right) ds+\left\{ \int_{\mathcal{O}}\left\vert \nabla V_{ml}\left( 0\right)
\right\vert ^{2}dx+T\, {\sum_{l< j\le m} \gamma_j^{-\theta}}\right\}}} .$$
We have proved above that $$\lim_{m \to \infty} \eta_{ml}  \leq   C(T) \left\vert \mathbb{E} \displaystyle{\int_{0}^{t}} \theta
_{l}(s) ds\right\vert  + C(T) \left\{\displaystyle{\int_{\mathcal{O}}}\left\vert b_{\sigma }^{0}-b_{\sigma
l}^{0}\right\vert ^{2}dx+{\displaystyle{\sum_{j=l+1}^{+\infty} \gamma_j^{-\theta}}}\right\}, $$
so that $$\displaystyle{\lim_{l\to \infty } \lim_{m \to \infty }} \eta_{ml} =0. $$

This completes the proof of Proposition \ref{P2}.
\end{proof}
\begin{proposition} \label{CS}
We have that,
\begin{enumerate}
 \item[(i)] $ G_\sigma(X_{m}) \rightarrow G_\sigma(X_\sigma)$ strongly in $L^2(\Omega \times \mathcal{O} \times (0,T) )$ as $m \rightarrow \infty$.
 
 \item[(ii)] ${\cal D}_{\sigma} (X_m) \rightarrow {\cal D}_\sigma(X_\sigma)$ strongly in $L^2(\Omega \times \mathcal{O} \times (0,T) )$ as $m \rightarrow \infty$.
 
 \item[(iii)] $h_\sigma (X_m) \rightarrow h_\sigma(X_\sigma)$ strongly in $L^2(\Omega \times \mathcal{O} \times (0,T) )$ as $m \rightarrow \infty$.

\end{enumerate} 
\end{proposition} 

\begin{proof}

We deduce from Proposition  \ref{P2} that there exists a sequence of positive numbers $\{\eta_{ml}\}$ such that 
\begin{equation*}
\mathbb{E}\int_{0}^{t}\int_{\mathcal{O}}\left(
D_{\sigma }\left( X_{m}\right) -D_{\sigma }\left( X_{l}\right) \right)
\left( X_{m}-X_{l}\right) dxds \leq \eta_{ml}
\end{equation*}
where 
$$
\eta_{ml} \rightarrow 0 \mbox{  as  } l, m  \rightarrow \infty \mbox{   for all   } m \geq l.
$$
First we prove i). In view of the definition of $G_\sigma$ \eqref{defG} we deduce that 
 \begin{align*}
\nonumber \mathbb{E} \int_0
^T \int_\mathcal{O}  (G_\sigma(X_{m}) - G_\sigma(X_{l}))^2 dxdt  =  \mathbb{E} \int_0
^T \int_\mathcal{O} ( \int_{X_{l}}^{{X_{m}}} \sqrt{{\cal D}_\sigma'(s)} ds )^2dx dt\\ 
 \nonumber  \leq\ \mathbb{E}\int_0
^T \int_\mathcal{O} ( \int_{X_{l}}^ {X_{m} } {\cal D}_\sigma'(s) ds ) ( \int_{X_{l} }^{X_{m}} 1 ds)dxdt \\
 \leq \mathbb{E} \int_0
^T \int_\mathcal{O} ( {\cal D}_\sigma(X_{m}) - {\cal D}_\sigma(X_{l} )  ( X_m - X_l ) dx dt.
 \end{align*}
 Thus 
\begin{equation*}
 \Vert G_\sigma(X_m ) - G_\sigma(X_l )\Vert^2_{L^2(\Omega \times \mathcal{O} \times (0,T))} 
 \leq \eta_{ml},
\end{equation*}
so that $G_\sigma(X_l )$ is a Cauchy sequence in $L^2(\Omega \times \mathcal{O} \times (0,T))$. Therefore, 
$$
G_\sigma(X_l ) \rightarrow \chi \mbox{   strongly in   } {L^2}(\Omega \times \mathcal{O} \times (0,T))\mbox{   as   } l \rightarrow \infty.
$$
Moreover, we recall that by \eqref{convxm} 
\begin{align*} 
	X_l\rightharpoonup X_\sigma ~~\mbox{weakly in}~~ L^2(\Omega \times \mathcal{O}  \times (0,T)  ) \mbox{as}~ l \to \infty. 
\end{align*}
Theorem \ref{Minty's trick} below then implies that 
\begin{equation*}
\chi = {G_\sigma}(X_\sigma).
\end{equation*}
Next, we prove $(ii)$.
Indeed, 
\begin{equation}\label{cauchyD}
\begin{array}{lll}
\nonumber \mathbb{E}\int_0
^T \int_\mathcal{O}  \vert {\cal D}_\sigma(X_m ) -  {\cal D}_\sigma(X_l) \vert^2dxdt &=& \mathbb{E}\int_0
^T \int_\mathcal{O} \vert \int_{X_l }^ {X_m } {\cal D}_\sigma'(r) d r\vert^2dxdt \\[10pt]
&\leq& \nonumber {\Vert {\cal D}_\sigma' \Vert}_{\infty} \mathbb{E} \int_0
^T \int_\mathcal{O}  \vert \int_ {X_l }^ {X_m } \sqrt{{\cal D}_\sigma'(r)} d r\vert^2dxdt \\[10pt]
& \leq & \nonumber  \max({d_1}, d_2)  ~ \mathbb{E} \int_0
^T \int_\mathcal{O}  \vert G_{\sigma}(X_l) - G_{\sigma}(X_m) \vert ^2dxdt\\[10pt]
& < &  max(d_1, d_2) \eta_{ml}.
 \end{array}
\end{equation}
As in the proof of (i), applying Proposition \ref{estXm} and Theorem  \ref{Minty's trick}  below, we deduce  that ${\cal D}_{\sigma}(X_l)$ strongly converges to its limit ${{\cal D}_\sigma}(X_\sigma)$ in   $L^2(\Omega \times \mathcal{O} \times (0,T) )$ as $l$ tends to infinity. \\

Finally we prove iii).
We deduce from \eqref{hsigma} that  
\begin{equation*} 
 \vert h_\sigma({X_m}) - h_\sigma({X_l }) \vert^2 \leq M \vert {\cal D}_\sigma(X_m) - {\cal D}_\sigma(X_l)\vert . \vert {b_\sigma^{-1} (X_m)} - { b_\sigma^{-1}(X_l)}\vert.
 \end{equation*}
  First we prove that $ \mathbb{E} \Vert h_\sigma (X_m) \Vert_{L^2(0,T;L^2(\cal O))} $    is bounded.
Indeed
  \begin{align*} 
	\nonumber 
	\mathbb{E}\int_0^T \int_\mathcal{O}  \vert h_\sigma({ X_m}) \vert^2dxdt = 
	\mathbb{E}\int_0^T \int_\mathcal{O}  \vert h_\sigma({ X_m}) - h_\sigma(0) \vert^2dxdt  \\ 
	\leq M \mathbb{E} \int_0^T \int_\mathcal{O} ({\cal D}_\sigma(X_m) - {\cal D}_\sigma(0)) ({b_\sigma^{-1} (X_m)} - { b_\sigma^{-1}(0)}) dxdt  \leq \mathbb{E} \int_0^T \int_\mathcal{O}  {{\cal D}_\sigma(X_m)} (X_m) dxdt \\
	\nonumber \leq\{\mathbb{E} \int_0^T \int_{\mathcal{O}}  {{\cal D}_\sigma(X_m)}^2 dxdt\}^\frac{1}{2}
	\{\mathbb{E} \int_0^T \int_\mathcal{O} (X_m)^2 dxdt\}^\frac{1}{2} \leq C {\tilde C},
\end{align*}
where we have used Proposition \ref{estXm}. 
  \begin{align*}
\nonumber 
\mathbb{E}\int_0^T \int_\mathcal{O}  \vert h_\sigma({ X_m}) - h_\sigma({X_l}) \vert^2dxdt  & \leq M \mathbb{E} \int_0
^T \int_\mathcal{O} \vert {\cal D}_\sigma(X_m) - {\cal D}_\sigma(X_l)\vert . \vert {b_\sigma^{-1} (X_m)} - { b_\sigma^{-1}(X_l)}\vert dxdt  \\
 \nonumber& \leq M \mathbb{E} \int_0
^T \int_\mathcal{O} ( {\cal D}_\sigma(X_m ) - {\cal D}_\sigma( X_l)) ( X_m- X_l )dxdt\\
& \leq M \eta_{ml}.
 \end{align*}
As in the proof of (i), applying Proposition \ref{estXm} and Theorem  \ref{Minty's trick}  below, we deduce  that $h_{\sigma}(X_l)$ strongly converges to its limit ${h_\sigma}(X_\sigma)$ in   $L^2(\Omega \times \mathcal{O} \times (0,T) )$ as $l$ tends to infinity. \\
\end{proof}

As we have seen above, Theorem \ref{Minty's trick} below is very useful to identify limit functions.

\begin{theorem} \label{Minty's trick}
Let $\Psi$ be a maximal monotone operator from $\mathbb{R}$ to $\mathbb{R}$
 with the properties  
\begin{enumerate} 
\item The sequence $U^{\varepsilon }$  weakly converges  to $U$ in  $L^2( \Omega \times \mathcal{O} \times (0,T) )$ as $\varepsilon$ tends to zero.
\item $\Psi(U^\varepsilon)$ strongly converges to $\Theta$ in $L^2(\Omega \times \mathcal{O} \times (0,T) )$ as $\varepsilon $ tends to zero. 
\end{enumerate}
Then $\Theta =\Psi(U)$. 
\end{theorem}

For the proof of this result, see Proposition 2.1 p.29 in \cite{BarbuNonlin},  or \cite{BrezisMM}.\\

\noindent Next let us prove the following Lemma which will be useful for passing to the limit as $m \to \infty$
\begin{lem}\label{projconv} There holds 
\begin{enumerate}
 \item[(i)] $P_m {\cal D}_{\sigma} (X_m) \rightarrow {\cal D}_\sigma(X_\sigma)$ strongly in $L^2(\Omega \times \mathcal{O} \times (0,T) )$ as $m \rightarrow \infty$.
 
 \item[(ii)] $P_m h_\sigma (X_m) \rightarrow h_\sigma(X_\sigma)$ strongly in $L^2(\Omega \times \mathcal{O} \times (0,T) )$ as $m \rightarrow \infty$.
 \end{enumerate}
\end{lem}
\begin{proof} We only present the proof of (i) since the proof of (ii) is similar. We have that
    \begin{align*} 
    &\Vert P_m  {\cal D}_{\sigma} (X_m) - {\cal D}_{\sigma} (X_\sigma) \Vert_{L^2(\Omega \times \mathcal{O} \times (0,T))} \\&
    \leq \Vert P_m  \left( {\cal D}_{\sigma} (X_m) - {\cal D}_{\sigma} (X_\sigma)\right)  \Vert_{L^2(\Omega \times \mathcal{O} \times (0,T))} + \Vert P_m  {\cal D}_{\sigma} (X_\sigma) - {\cal D}_{\sigma} (X_\sigma) \Vert_{L^2(\Omega \times \mathcal{O} \times (0,T))}\\&
    \leq c \Vert {\cal D}_{\sigma} (X_m)- {\cal D}_{\sigma} (X_{\sigma}) \Vert_{L^2(\Omega \times \mathcal{O} \times (0,T))} + \Vert P_m {\cal D}_{\sigma} (X_{\sigma}) - {\cal D}_{\sigma} (X_{\sigma})\Vert_{L^2(\Omega \times \mathcal{O} \times (0,T))}.
    \end{align*}

We deduce from Proposition
\ref{CS} that the first term of the right-hand-side converges to zero and from \eqref{conv} that the second term also converges to zero. 
\end{proof}

The final step is to pass to the limit as $m\to \infty$. To that purpose, we
multiply \eqref{pbinm} by the product $y \psi$, where $y(\omega)$ is any
a.s. bounded random variable and  $\psi(t)$ is a bounded function on $(0,T)$%
; we integrate from $0$ to $T$ and we take the expectation, which yields for
all $j = 1,...,m$

\begin{align}  \label{passlim}
\mathbb{E}\int_0^T y \psi &\int_ \mathcal{O} X_m(x,t) \Db{e}_j dx dt  \notag \\& = \mathbb{E}%
\int_0^T y \psi \int_ \mathcal{O}\displaystyle{\ b^0_{\sigma m}(x)} \Db{e}_j dx
dt + \mathbb{E}\int_0^T y \psi \int_0^t \int_\mathcal{O} P_m {D}_{\sigma}
( X_m) \Delta \Db{e}_j dx dt  \notag \\
& \nonumber + \mathbb{E}\int_0^T y \psi \int_0^t \int_D P_m h_\sigma(X_m) \Db{e}_j dx dt + 
\mathbb{E}\int_0^T y \psi \Db{\beta}_j(t
)\sqrt{\lambda_j} dx dt \\&= \mathbb{E}%
\int_0^T y \psi \int_ \mathcal{O}\displaystyle{\ b^0_{\sigma m}(x)} \Db{e}_j dx
dt + \mathbb{E}\int_0^T y \psi \int_0^t \int_\mathcal{O} P_m {D}_{\sigma}
( X_m) \Delta \Db{e}_j dx dt  \notag \\
&+ \mathbb{E}\int_0^T y \psi \int_0^t \int_D P_m h_\sigma(X_m) \Db{e}_j dx dt + 
\mathbb{E}\int_0^T y \psi \int_ \mathcal{O} P_m(\sqrt{Q} W(t))e_j  dx dt
\end{align}

\noindent Passing to the limit in \eqref{passlim} by applying Lebesgue dominated
convergence Theorem, using \eqref{pbinm0}, \eqref{convxm} and Lemma %
\ref{projconv}, we deduce that

\begin{align*}
\mathbb{E}\int_0^T \int_ \mathcal{O} y \psi X_\sigma \Db{e}_j dx dt = \mathbb{E}%
\int_0^T \int_ \mathcal{O}y \psi \displaystyle{\ b^0_{\sigma }(x)} \Db{e}_j dx dt
+ \mathbb{E}\int_0^T y \psi \int_0^t \int_\mathcal{O}  {D}_{\sigma} (
X_\sigma) \Delta  \Db{e}_j dx dt  \notag \\
+ \mathbb{E}\int_0^T y \psi \int_0^t \int_D h_\sigma(X_\sigma) \Db{e}_j dx dt + {\mathbb{E}\int_0^T y \psi \int_ \mathcal{O} \sqrt{Q} W(t)e_j  dx dt}
\end{align*}
for all $y \in L^\infty(\Omega)$ and $\psi \in L^{\infty}(0, T)$.\newline

We remark that the linear combinations of the ${e}_j$'s are dense in $H^1_0(%
\mathcal{O})$, so that  
\begin{align*}
\mathbb{E}\int_0^T \int_ \mathcal{O} y \psi X_\sigma w dx dt = \mathbb{E}%
\int_0^T \int_ \mathcal{O}y \psi \displaystyle{\ b^0_{\sigma }(x)} w dx dt + 
\mathbb{E}\int_0^T y \psi \int_0^t \dual {\Delta {D}_{\sigma} ( X_\sigma),
 w }  dt  \notag \\
+ \mathbb{E}\int_0^T y \psi \int_0^t \int_D h_\sigma(X_\sigma) w dx dt + \sum_{j=1}
^\infty \mathbb{E}%
\int_0^T y \psi \int_\mathcal{O} \sqrt{Q}W(t) w  dx dt
\end{align*}
for all $y \in L^\infty(\Omega)$, $\psi \in L^{\infty}(0, T)$ and $w \in H^1_0( \mathcal{O})$.\\
Thus we  conclude that almost surely and for a.e. $t \in (0,T)$  
\begin{equation*}
{X_{\sigma}}(t)  = b^0_{\sigma} + \int_0^t \{ \Delta {D}_{\sigma
}\left( X_{\sigma}(s)\right) + h_{\sigma }  \left( X_{\sigma}(s) \right) \}
ds + {\sqrt{Q}W(t),}\ \mathnormal{in }\  H^{-1}(\mathcal{O}).
\end{equation*}

\subsection{Proof of the uniqueness}

As for the existence proof, the main idea is to work in the ${H^{-1}}(%
\mathcal{O})$ topology. This will also be the key idea for obtaining the
error estimate between $X_\sigma$ and $X_{0}$. \newline

Next we prove the uniqueness of the solution of Problem $(P_\sigma)$.

\begin{lem}
\textit{The weak solution of Problem $(P_\sigma)$ is unique.}
\end{lem}

\noindent \textbf{Proof } Let $X^1_{\sigma}$ and $X^2_{\sigma}$ be two
solutions of Problem $(P_\sigma)$. Then, almost surely, and for a.e. $t \in
(0,T),$

\begin{equation*}
(X^1_{\sigma} - X^2_{\sigma})(t) = \int_0^t \{ \Delta D_{\sigma }\left(
X^1_{\sigma}\right) - \Delta D_{\sigma }\left( X^2_{\sigma}\right) +
h_{\sigma} \left( X^1_{\sigma}\right) - h_{\sigma} \left(
X^2_{\sigma}\right)\} ds \mbox{   in   } {H^{-1}} (\mathcal{D}). 
\end{equation*}
Setting $F(X) = {\lVert X \rVert }^2_{{H^{-1}}(\mathcal{O})}$, and noting
that $\displaystyle{\frac{\partial F}{\partial X}}(X) = 2 X$, we deduce that 
\begin{equation*}
\lVert {\ (X^1_{\sigma} - X^2_{\sigma})(t) \rVert }^2_{{H^{-1}}(\mathcal{O}%
)} = \int_0^t < \{ \Delta D_{\sigma }\left( X^1_{\sigma}\right)- \Delta
D_{\sigma }\left( X^2_{\sigma}\right) + h_{\sigma} \left(
X^1_{\sigma}\right) - h_{\sigma} \left( X^2_{\sigma}\right)\}, 2
(X^1_{\sigma} - X^2_{\sigma})> ds 
\end{equation*}
Next we set $v_1 = - \Delta ^{-1} X^1_{\sigma }$, $v_2 = - \Delta ^{-1}
X^2_{\sigma }$ and $v=v_1-v_2$. 

This yields 
\begin{equation*}
\begin{array}{cc}
{\frac{1}{2}} \lVert {\nabla (v_1 - v_2)} \rVert^2(t)& + \int_0^t \int_%
\mathcal{O} ({\cal D}_{\sigma}( X^1_{\sigma}) - {\cal D}_{\sigma} ( X^2_{\sigma }))
(X^1_{\sigma} - X^2_{\sigma}) dxds  \\[10pt]&
= - \int_0^t \int_\mathcal{O}
(h_{\sigma}\left( X^1_{\sigma } \right)- h_{\sigma }( X^2_{\sigma } ) )(v_1
- v_2) dxds,
\end{array}
\end{equation*}
for a.e. $t\in (0,T).$ In view of \eqref{hsigma} and the fact that $\vert
(b^{-1}_\sigma(X^1_{\sigma }) - b^{-1}_\sigma(X^2_{\sigma}) \vert \leq \vert
X^1_\sigma - X^2_\sigma \vert $, we deduce that 
\begin{align}  \label{hyphu}
& & \vert \int_\mathcal{O} (h_{\sigma }\left( X^1_{\sigma } \right)-
h_{\sigma }( X^2_{\sigma } ) )( v_1-v_2) dx \vert \leq \frac{\varepsilon }{2}%
\int_{\mathcal{O}} \big(h_{\sigma }\left( X^1_{\sigma } \right)- h_{\sigma
}( X^2_{\sigma } ) \big)^2dx + \frac{1}{2 \varepsilon}\int_{\mathcal{O}} (
v_1-v_2)^2 dx  \notag \\
& & \leq \frac{\varepsilon M }{2}\int_{\mathcal{O}} \vert {\cal D}_{\sigma}(
X^1_{\sigma}) - D_{\sigma }\left( X^2_{\sigma } \right) \vert \vert
b^{-1}_\sigma(X^1_{\sigma }) - b^{-1}_\sigma(X_\sigma^2) \vert dx + \frac{1%
} { 2 \varepsilon} \Vert v_1 -v_2 \Vert^2_{H^1_0}  \notag \\
& & \leq \frac{\varepsilon M }{2}\int_{\mathbb{D}} \vert {\cal D}_{\sigma} \left(
X^1_{\sigma} \right) - D_{\sigma } \left( X^2_{\sigma } \right) \vert \vert
X^1_{\sigma } - X_\sigma^2 \vert dx + \frac{1} { 2 \varepsilon} \Vert v_1
-v_2 \Vert^2_{H^1_0}
\end{align}
Then, by choosing $\varepsilon = \displaystyle{\frac{1}{M}}$ in \eqref{hyphu}
we deduce that 
\begin{equation*}
\Vert (v_1 -v_2) (t) \Vert_{H^1_0} ^2 + \int_0^t \int_{\mathcal{O}} (
{\cal D}_{\sigma}( X^1_{\sigma}) - D_{\sigma }\left( X^2_{\sigma } \right)
(X^1_{\sigma } - X_\sigma^2) dx \leq M \int_0^t \Vert v_1- v_2
\Vert^2_{H^1_0},
\end{equation*}
which in view of the monotonicity of ${\cal D}_{\sigma}$ implies that 
\begin{equation*}
\Vert (v_1 -v_2) (t) \Vert_{H^1_0} ^2 \leq M \int_0^t \Vert v_1- v_2
\Vert^2_{H^1_0},
\end{equation*}
Thus $v_1=v_2$, so that also $X_\sigma^1=X^2_\sigma$.

\section{Singular limit as $\protect\sigma $ tends to zero.}

In this section, we prove an error estimate between $X_\sigma$ and $X_0$.

\begin{theorem}
The following error estimate holds  
\begin{equation*}
\Vert X_\sigma - X_0 \Vert_{L^2( \Omega \times (0,T) \times \mathcal{O})}
\leq C \sqrt{\sigma},
\end{equation*}
where $C$ is a positive constant independent of $\sigma >0$.
\end{theorem}

\noindent \textbf{Proof } We take the difference of the equations for $%
X_\sigma$ and $X_{0}$ to deduce that almost surely and for a.e. $t \in (0,T)$
\begin{equation*}
(X_{0} - X_\sigma)(t) = \int_0^t \{ \Delta D \left( X_{0}\right) - \Delta
D_{\sigma }\left( X_\sigma\right) + h \left( X_{0}\right) - h_{\sigma}
\left( X_\sigma\right) \} ds \mbox{   in   } {H^{-1}} (\mathcal{O}). 
\end{equation*}
%We choose the test function $ \varphi = v $ such that $v= - \Delta^{-1}w$ in $H^1_0(\mathcal{O})$ where 
%$w=   X_{0} - X_\sigma$.

Since $F(X) = {\lVert X \rVert }^2_{{H^{-1}}(\mathcal{O})}$, and noting that 
$\frac{\partial F}{\partial X}(X) = 2 X$, we deduce that 
\begin{align*}
&\lVert {\ (X_{0} - X_\sigma)(t) \rVert }^2_{{H^{-1}}(\mathcal{O})} - {\
\lVert (X_{0} - X_\sigma)(0) \rVert }^2_{{H^{-1}}(\mathcal{O})} \\
=& \int_0^t < \{ \Delta D\left( X_{0}\right) - \Delta D_{\sigma }\left(
X_\sigma\right) + h\left( X_{0}\right) - h_{\sigma} \left(
X_\sigma\right)\}, 2 (X_{0} - X_\sigma)> ds.
\end{align*}
Next we set $v_0 = - \Delta ^{-1} X_{0}$, $v_{\sigma } = - \Delta ^{-1}
X_\sigma$ and $v=v_0-v_{\sigma}$. This yields 
\begin{equation*} 
\begin{split}
&{\frac{1}{2}} \lVert {\nabla (v)} \rVert^2(t) - {\frac{1}{2}} \lVert {%
\nabla (v)} \rVert^2(0) + \int_0^t \int_\mathcal{O} (D_{0}( X_{0}) - D_{0} (
X_{\sigma })) (X_{0} - X_\sigma)dxds = \\
&-\int_0^t \int_\mathcal{O} \{ ( D ( X_{\sigma }) - {\cal D}_{\sigma}( X_\sigma)
(X_{0} - X_\sigma) + (h \left( X_{0}\right) - h\left( X_\sigma\right) + h(
X_\sigma) - h_{\sigma }( X_{\sigma }) )v\}dxds,
\end{split}%
\end{equation*}
for a.e. $t\in (0,T)$. Taking the expectation yields 
\begin{equation}
\begin{split}  \label{inter}
\frac 12 \mathbb{E} \Vert v (t)\Vert^2_{H^1_0(\mathcal{O})}& + \mathbb{E}%
\int_0^t \int_{\mathcal{O}} [ D(X_{0} ) - D (X_\sigma) ] ( X_{0} - X_\sigma)
dxds \\
& \leq \frac12 \mathbb{E} \Vert v(0)\Vert^2 _{H^1_0(\mathcal{O})} + \mathbb{E%
}\int_0^t \int_{\mathcal{O}} \vert D(X_\sigma ) - {\cal D}_{\sigma} (X_\sigma)
\vert \vert X_{0} -X_\sigma \vert dxds \\
& + \mathbb{E}\int_0^t \int_{\mathcal{O}} (h(X_{0}) - h( X_\sigma) ) v dxds
+ \mathbb{E}\int_0^t \int_{\mathcal{O}} (h(X_{\sigma }) - h_{\sigma}(
X_\sigma) ) v dxds.
\end{split}%
\end{equation}
First we consider the second term on the right-hand-side of \eqref{inter}.
Since 
\begin{equation*}
\vert D(s) - {\cal D}_{\sigma}(s)\vert =\left\{ 
\begin{array}{ll}
0, & s\leq 0, \\ 
d_1 s, & 0 < s < \sigma, \\ 
d_{1}\sigma , & s \geq \sigma ,%
\end{array}
\right.
\end{equation*}
we have that in view of Proposition \ref{estXm}

\begin{equation}
\begin{split}
\mathbb{E}\int_0^t \int_{\mathcal{O}} \vert D(X_\sigma ) - {\cal D}_{\sigma}
(X_\sigma) \vert \vert X_{0} -X_\sigma \vert dxds & \leq d_1 \vert \sigma
\vert \mathbb{E} \int_0^t \int_\mathcal{O} \vert X_{0} - X_\sigma \vert dxds
\notag \\
& = d_1 \vert \sigma \vert\int_0^t (\int_\Omega \int_\mathcal{O} \vert X_{0}
- X_\sigma \vert^2)^{\frac12} ds \vert \mathcal{O}\vert^{\frac12}  \notag \\
& \leq d_1 \vert \sigma\vert \vert \mathcal{O}\vert^{\frac12} \sup_{\tau \in
(0,T)} \Vert X_0 (\tau) - X_\sigma(\tau) \Vert_{L^2(\Omega \times \mathcal{O}%
)}  \notag \\
& \leq 2 C d_1 \vert \sigma\vert \vert \mathcal{O}\vert^{\frac12}  \notag
\end{split}%
\end{equation}
where $\vert\mathcal{O}\vert$ stands for the Lebesgue measure of $\mathcal{O}
$. \newline
As for the third term on the right-hand-side of \eqref{inter}, there holds 
\begin{equation}
\begin{split}
\mathbb{E} \int_0^t \int_\mathcal{O} ( h(X^{0 }) - h( X_{\sigma }) )v dsdx &
\leq \frac{\varepsilon}{2} \mathbb{E} \int_0^t \int_\mathcal{O} ( (h(X_{0})
- h( X_{\sigma }) )^2dsdx + \frac{1}{2\varepsilon}\mathbb{E} \int_0^t\Vert
v\Vert_{H^1_0(\mathcal{O})}^2ds  \notag \\
 \leq \frac{\varepsilon M}{2}\mathbb{E}\int_0^t \int_{\mathcal{O}}
&\max(d_1,d_2) \vert X_{0} - X_\sigma \vert^2dsdx + \frac{1}{2\varepsilon}%
\mathbb{E} \int_0^t\Vert v\Vert_{H^1_0(\mathcal{O})}^2 ds,  \notag
\end{split}%
\end{equation}
where we have used \eqref{hyph1}. As for the fourth term, we have 
\begin{equation*}
\vert h(r) - h_\sigma(r)\vert =\left\{ 
\begin{array}{ll}
0, & r<0, \\ 
h(r) - h(0), & 0 < r < \sigma, \\ 
h(r)-h(r-\sigma) , & r>\sigma ,%
\end{array}
\right.
\end{equation*}
so that, in view of \eqref{hyph1}, 
\begin{equation*}
\vert h(r)- h_\sigma(r) \vert \leq \sqrt{M max(d_1,d_2)} \sigma,
\end{equation*}
which in turn implies that 
\begin{align}
\mathbb{E}\int_0^t \int_{\mathcal{O}} (h(X_{\sigma }) - h_{\sigma}(
X_\sigma) ) v dxds \leq \frac{\varepsilon M }{2} \max(d_1,d_2) \sigma^2 T
\vert \mathcal{O}\vert + \frac{1}{2 \varepsilon} \mathbb{E} \int_0^t \Vert v
\Vert^2_{H^1_0(\mathcal{O})}ds.  \notag
\end{align}
As for the initial function, we have that 
\begin{equation*}
v(0) = - \Delta^{-1}( b(u_0) - b_\sigma(u_0)) = - \Delta^{-1} \chi _0,
\end{equation*}
where $\chi _0 := b(u_0) - b_\sigma(u_0)$. Remark that $0 \leq \chi_0 <
\sigma $\, which implies that $\Vert v(0) \Vert_{H^1_0(\mathcal{O})} \leq C
\sigma. $ Collecting all the terms and noting that ${D}^{\prime }_0 \geq
\min (d_1,d_2)$ yields 
\begin{equation}
\begin{split}
\frac 12 \mathbb{E} \Vert v(t) \Vert^2_{H^1_0(\mathcal{O})} + & \min
(d_1,d_2) \mathbb{E}\int_0^t \int_{\mathcal{O}} ( X_{0} - X_\sigma)^2 dxds \\
\leq & \;\; \frac12 C \sigma^2 +2 C d_1 \vert \sigma\vert \vert \mathcal{O}%
\vert^{\frac12} +\frac{\varepsilon M}{2}\mathbb{E} \int_0^t \int_{\mathcal{O}%
} \max(d_1,d_2) (X_{0} - X_\sigma )^2dsdx \\&+ \frac{1}{2\varepsilon}\mathbb{E}
\int_0^t\Vert v\Vert_{H^1_0(\mathcal{O})}^2ds  \notag 
 + \frac{\varepsilon M }{2} \max(d_1,d_2) \sigma^2 T \vert \mathcal{O}\vert
+ \frac{1}{2 \varepsilon} \mathbb{E} \int_0^t \Vert v \Vert^2_{H^1_0(%
\mathcal{O})}ds.
\end{split}%
\end{equation}
Thus, 
\begin{equation}  \label{gronwall}
\begin{split}
\frac 12 \mathbb{E} \Vert v(t) \Vert^2_{H^1_0(\mathcal{O})} &+ \left(\min
(d_1,d_2) - \frac{\varepsilon M}{2}\max(d_1,d_2) \right)\mathbb{E}\int_0^t
\int_{\mathcal{O}} ( X_{0} - X_\sigma)^2 dxds \\
& \leq \frac12 C \sigma^2 + 2 C d_1 \vert \sigma\vert \vert \mathcal{O}%
\vert^{\frac12} + \frac{\varepsilon M }{2} \max(d_1,d_2) \sigma^2 T \vert 
\mathcal{O}\vert + \frac{1}{ \varepsilon} \mathbb{E} \int_0^t \Vert v
\Vert^2_{H^1_0(\mathcal{O})}ds.
\end{split}%
\end{equation}
We choose $\varepsilon$ small enough. In order to apply Gronwall Lemma, we
rewrite the equation in the form 
\begin{align}
\mathbb{E} \Vert v(t) \Vert^2_{H^1_0(\mathcal{O})} \leq c \sigma + \tilde{C}
\int_0^t \mathbb{E} \Vert v(t) \Vert^2_{H^1_0(\mathcal{O})} ds,  \notag
\end{align}
so that, $\mathbb{E} \Vert v(t) \Vert^2_{H^1_0(\mathcal{O})} \leq c \sigma
e^{\tilde{C}t}.$ In particular, as $\sigma\rightarrow 0$, $v \to 0 $ in $%
L^\infty (0,T; L^2(\Omega;H^1_0(\mathcal{O}))) $.

Finally we deduce from \eqref{gronwall} that $\Vert X_\sigma - X_{0}
\Vert_{L^2( \Omega \times (0,T) \times \mathcal{O})} \leq C \sqrt{\sigma}.$ 

\section{The deterministic case}

Let us now consider the deterministic case, where the function $h$ coincides with that from the 
biological framework proposed by Mimura (cf. Hilhorst, Salin, Schneider, Gao \cite{HSSG}). The 
function $h$ is a quadratic function defined as follows. We set 
$$
f(r) = \lambda r(1-r), \,\,\, g(r) = \mu r(1-r) 
$$ 
for $r \in [-1,1]$ and 
\begin{equation*}
	h\left( r\right) =\left\{ 
	\begin{array}{ll}
		g(-r) & r \in [-1,0], \\ 
		f(r) , & r \in [0,1].%
	\end{array}
	\right.
\end{equation*}
Then ${B_\sigma} = 0$ so that Problem $(P_\sigma)$ reads as 
\begin{equation*}
	(P_\sigma) \quad \quad \left\{ 
	\begin{array}{ll}
		\displaystyle{ \frac{{\partial}X_{\sigma}}{{\partial}t}}= \Delta {D}_{\sigma }\left( X_{\sigma }\right) +h_{\sigma
		}\left( X_{\sigma }\right), & ~\mathcal{O}\times \left(
		0,T\right) , \\ 
		X_{\sigma }=0, & ~\partial \mathcal{O}\times \left( 0,T\right) , \\ 
		X_{\sigma }\left( 0\right) =b_{\sigma }^{0}, & ~\mathcal{O}\times \left\{
		0\right\} , %
	\end{array}
	\right. 
\end{equation*}
where ${\cal D}_\sigma$, $h_\sigma$ and $b^0_\sigma$ are defined as before. We proved that its unique weak solution $X_\sigma$ is such that (cf. \cite{HMS}, \cite{HSSG})
$$
-1 \leq {X_\sigma} \leq 1+ \sigma.
$$
so that the functtion ${b^{-1}_\sigma}(X_\sigma)$ satisties
$$
-1 \leq {b^{-1}_\sigma}(X_\sigma) \leq 1.
$$
The function $h^{'}$ is bounded on the interval [-1,1] so that $h$ is Lipschitz continuous on that interval. 
We can then define a Lipschitz continuous extension of  the function $h$ on the whole real line so that
it satisfies the hypotheses which we set earlier. Therefore the following error estimate holds 
\begin{equation*}
	\Vert X_\sigma - X_{0} \Vert_{L^2(\mathcal{O} \times (0,T))} \leq C \sqrt{\sigma}. \\
\end{equation*}

\medskip

\section{Appendix}

\begin{proof}[Proof of Lemma \protect\ref{Pmdelta}]

First we prove (i).

\begin{align*}
    \int_{\mathcal{O}} (P_mA) (x)X_m(x) dx & =   \int_{\mathcal{O}} \left(\sum_{j=1}^m  \dual{ A, \Db{e}_j} \Db{e}_j(x)\right) X_m(x) dx \\
    &= \int_{\mathcal{O}} \left(\sum_{j=1}^m  \dual{ A, \Db{e}_j} \Db{e}_j(x)\right) \left( \sum_{k=1}^m X_k e_k(x) \right) dx \\
    &=  \sum_{j=1}^m \sum_{k=1}^m \int_{\mathcal{O}} \left(  \dual{ A, \Db{e}_j} \Db{e}_j(x)  X_k e_k(x) \right) dx \\
    &=  \sum_{j=1}^m \sum_{k=1}^m \dual{ A, \Db{e}_j}   X_k \int_{\mathcal{O}}  e_j(x)  e_k(x)  dx 
    \\&=  \sum_{j=1}^m\dual{ A, e_j}   X_j 
\end{align*}
On the other hand, $\dual{A, X_m} = \dual{A,\sum_{j=1}^m  e_j X_j } = \sum_{j=1}^m \dual{A, e_j} X_j.   $\\
We  prove (ii). We have that 
\begin{equation*}
\begin{split}
P_m \Delta A = \sum_{j=1}^m  (\int_\mathcal{O} \Delta A(x) \Db{e}_j(x) dx) \Db{e}_j = \sum_{j=1}^m  (\int_\mathcal{O} A(x) \Delta \Db{e}_j(x) dx) \Db{e}_j\\
= - \sum_{j=1}^m  \lambda_j(\int_\mathcal{O} A(x) \Db{e}_j(x) dx) \Db{e}_j
\end{split}
\end{equation*}
and that
\begin{equation*}
\begin{split}
\Delta (P_m A) = \Delta \sum_{j=1}^m (\int_\mathcal{O}  A(x) \Db{e}_j(x) dx) \Db{e}_j = \sum_{j=1}^m  (\int_\mathcal{O} A(x) \Db{e}_j(x) dx)\Delta \Db{e}_j\\
= - \sum_{j=1}^m  \lambda_j(\int_\mathcal{O} A(x) \Db{e}_j(x) dx) \Db{e}_j.
\end{split}
\end{equation*}

\noindent Next we prove (iii). We have that 

$$ P_m A = \sum_{j=1}^m \left( \int_{\mathcal{O}} A(x)\Db{e}_j(x) dx \right) \Db{e}_j  $$
$$ \Delta^{-1} (P_m A) = \sum_{j=1}^m  \left( \int_{\mathcal{O}} A(x)\Db{e}_j(x) dx \right) \Delta^{-1} \Db{e}_j = - \frac{1}{\lambda_j} \sum_{j=1}^m \left( \int_{\mathcal{O}} A(x)\Db{e}_j(x) dx \right) \Db{e}_j $$
\begin{equation*}
\begin{split}
P_m  \{ \Delta^{-1} A \} &= \sum_{j=1}^m  \left(\int_\mathcal{O} {\Delta^{-1}} A(x) \Db{e}_j(x) dx \right) \Db{e}_j\\
& = \sum_{j=1}^m  \left( \int_\mathcal{O}  A(x)  {\Delta^{-1}} \Db{e}_j(x) dx\right) \Db{e}_j\\
& = - \frac{1}{\lambda_j} \sum_{j=1}^m \left( \int_{\mathcal{O}} A(x) \Db{e}_j(x) dx \right) \Db{e}_j.
\end{split}
\end{equation*}
\noindent Finally for (iv) we can easily see that 
$$ \int_{\mathcal{O}} (P_m A(x))B(x)dx = \int_{\mathcal{O}} A(x) (P_m B(x))dx$$
since $P_m$ is self-adjoint.
\end{proof}

\noindent \textbf{Acknowledgements} The authors would like to thank
Professor Guy Vallet and Professor Kunwoo Kim for helpful discussions.
D.G. aknowledges financial support from National Sciences and Engineering Research Council (NSERC), Canada, 
Grant/Award Number: RGPIN-2025-03963.
I.C. was partially supported by the Agence Nationale de la Recherche (ANR), project ANR-22-
CE40-0010 COSS and by Region Normandie and ERDF fund via the Idemo Scale Op project under convention
00152289 (BPI France and France 2030 AAP).


\begin{thebibliography}{9}

\bibitem{BarbuNonlin} Viorel Barbu, \textit{Nonlinear Differential Equations of Monotone Types in Banach Spaces}, Spronger, 2010.

\bibitem{BII}Viorel Barbu, Ioana Ciotir, Ionut Danaila, \textit{Existence and uniqueness of solution to the two-phase Stefan problem with convection},
Applied Mathematics and Optimization, 84(2) (2021).

\bibitem{BDPR}  {Viorel Barbu, Giuseppe Da Prato, Michael R\"ockner, \textit{Stochastic porous media equations},  Lecture Notes in Math. 2163, Springer,
2016.}

\bibitem{BDP}
Viorel Barbu and Giuseppe Da~Prato, \emph{The two phase stochastic {S}tefan
  problem}, Probab. Theory Related Fields, 124 (2002), no.~4, 544-560.

\bibitem{Rd} Viorel Barbu, Michael R\"{o}ckner, Francesco Russo, \textit{Stochastic porous media equations in $\mathbb{R}^{d}$}. J. Math. Pures Appl., 2015, 103(4): 1024-1052.

\bibitem{francesco2} Viorel Barbu, Michael R\"{o}ckner, Francesco Russo, \textit{Doubly probabilistic representation for the
stochastic porous media type equation}. Annales de l’Institut Henri Poincaré. Section Probabilités et Statistiques, vol. 53 No. 4, pp. 2043-2073, nov, 2017.


\bibitem{BrezisMM} {Haim Brezis, \textit{Operateurs maximaux monotones et semi-groupes de contractions dans les espaces de Hilbert}, Amsterdam by North-Holland Pub Co, 1973.}

\bibitem{Ciotir} Ioana Ciotir,  \textit{A Trotter type result for the stochastic
porous media equations}, Nonlinear Anal., Theory Methods Appl., Ser. A,
Theory Methods 71, No. 11, (2009) 5606-5615.



\bibitem{eu-conv} Ioana Ciotir, \textit{Convergence of the solutions for the
stochastic porous media equations and homogenization}, Journal of Evolution
Equation, 11 (2011), 339-370.

\bibitem{eu-conv2} Ioana Ciotir, \textit{A Trotter-type theorem for nonlinear stochastic equations in variational formulation and homogenization}, Differential and Integral Equations, Khayyam Publishing Company, ISSN 0893 - 4983, Volume 24, Issue 3-4, (2011) p. 371-388.

\bibitem{franco1} Ioana Ciotir and Franco Flandoli and Dan Goreac, \textit{An Existence Result for a Stochastic Stefan Problem With Mushy Region and Turbulent Transport Noise} arXiv preprint arXiv:2505.08500 (2025).



\bibitem{franco2} Ioana Ciotir and Franco Flandoli and Dan Goreac, \textit{The Stefan problem with mushy region as a scaling limit of stochastic PDE with turbulent transport}, Journal of Dynamics and Differential Equations, January 2026.

\bibitem{reika} Ioana Ciotir, Reika Fukuizumi, Dan Goreac, \textit{The stochastic fast logarithmic equation in Rd with multiplicative Stratonovich noise}. Journal of Mathematical Analysis and Applications, 542(1), 128786, 2025.

\bibitem{IDS} Ioana Ciotir, Dan Goreac, Juan Li et al. \textit{A stochastic porous media Schrödinger equation: Feynman-type motivation, well-posedness and control interpretation.} J. Evol. Equ. 25, 5 (2025).

\bibitem{anna} Ioana Ciotir, Dan Goreac, Anna-Mariya Otsetova, \textit{Two-Phase Stefan Problem with Convection and Stochastic Noise}, to appear in Irregular Stochastic Analysis, Springer series.


\bibitem{DZ}  Giuseppe Da Prato and Jerzy Zabczyk, \textit{Stochastic
equations in infinite dimensions}, Encyclopedia Math. Appl., 152, Cambridge
University Press, Cambridge, 2014.

\bibitem{EK} {Perla El Kettani, 
 Well-posedness of a stochastic phase-field problem with multiplicative noises, Math. Methods Appl. Sci. 43 (2020), No. 15, 8538–856.}

\bibitem{HMS} {Danielle Hilhorst, Masayasu Mimura, and Reiner Sch\"atzle,
Vanishing latent heat limit in a Stefan-like problem arising in biology,
Nonlinear Anal., Real World Appl. 4, No. 2, (2003) 261-285.}

\bibitem{HSSG}{Danielle Hilhorst, Florian Salin, Victor Schneider, and Yueyuan Gao, 
	Lecture notes on the singular limit of reaction-diffusion systems, 
	Interdiscip. Inf. Sci. 29, No. 1, (2023) 1-53.}

\bibitem{Karatzas} Ioannis Karatzas and Steven Shreve, \textit{Brownian
motion and stochastic calculus}, Vol. 113, Springer Science and Business
Media, 1991.

\bibitem{PR} Claudia Prévôt and Michael Röckner,  \textit{A Concise Course 
on Stochastic Differential Equations}, Lecture Notes in Math. 1905, Springer, Berlin, 2007.

\bibitem{Edi1} Aurel Răşcanu, Eduard Rotenstein, \textit{The Fitzpatrick function - a bridge between convex analysis and multivalued stochastic differential equations}, J. Convex Anal., no. 18, no. 1, pp. 105-138, 2011.


\bibitem{Edi2} Aurel Răşcanu, Eduard Rotenstein, \textit{Obstacle problems for parabolic SDEs with Hölder continuous diffusion: from weak to strong solutions}, J. Math. Anal. Appl., Volume 450, Issue 1 (June, 1), pp. 647-669, 2017.

\bibitem{francesco} Michael Röckner and Francesco Russo, \textit{Uniqueness for a class of stochastic Fokker-Planck and
porous media equations}. Journal of Evolution Equations, Springer-Verlag, Journal
of Evolution Equations, vol. 17 (3), pp. 1049-1062, Springer-Verlag, oct, 2017.

\bibitem{Vallet}  {Guy Vallet, Stochastic perturbation of nonlinear
degenerate parabolic problems, Differ. Integral Equ. 21, No. 11-12, (2008)
1055-1082.}
\end{thebibliography}
\end{document}